\documentclass[a4paper,11pt]{article}
\usepackage[utf8x]{inputenc}
\usepackage{authblk}

\usepackage{verbatim}
\usepackage{amssymb}
\usepackage{amsthm}
\usepackage{amsmath}
\usepackage[english]{babel}
\usepackage{accents}
\usepackage{amsfonts}
\usepackage{textcomp}
\usepackage{stmaryrd}
\usepackage{tikz}
\usepackage{booktabs}
\usepackage{url}
\usepackage[scaled=0.90]{helvet}
\usepackage{graphicx}
\graphicspath{ {/home/stefan/Pictures/} }

\title{NS saturated and $\Delta_1$-definable} 
\author{Stefan Hoffelner\footnote{The author was supported by FWF-GA\v CR grant no. 
17-33849L, Filters, ultrafilters and connections with forcing. Additional funding by the Deutsche Forschungsgemeinschaft (DFG German Research Foundation) under Germanys Excellence Strategy EXC 2044 390685587, Mathematics M\"unster: Dynamics-Geometry-Structure. }}

\affil{Institut f\"ur Mathematische Logik und Grundlagenforschung \\WWU M\"unster}
\date{23.06.2017}

\begin{document}
\maketitle

\newtheorem{thm}{Theorem}
\newtheorem{Claim}[thm]{Claim}
\newtheorem{Fact}[thm]{Fact}
\newtheorem{Proposition}[thm]{Proposition}
\newtheorem{Definition}[thm]{Definition}
\newtheorem{Assumption}[thm]{Assumption}
\newtheorem{Lemma}[thm]{Lemma}
\newtheorem{Question}{Question}
\newtheorem{Corollary}[thm]{Corollary}
\newtheorem*{Theorem*}{Theorem}

\newcommand{\ZFC}{\mathsf{ZFC}}
\newcommand{\ZF}{\mathsf{ZF}}
\newcommand{\ZFP}{\mathsf{ZFC}^-}
\newcommand{\AC}{\mathsf{AC}}
\newcommand{\PFA}{\mathsf{PFA}}
\newcommand{\BPFA}{\mathsf{BPFA}}
\newcommand{\MRP}{\mathsf{MRP}}
\newcommand{\MM}{\mathsf{MM}}
\newcommand{\BMM}{\mathsf{BMM}}
\newcommand{\MA}{\mathsf{MA}_{\omega_1}}

\newcommand{\SCC}{\mathsf{SCC}}
\newcommand{\CC}{\mathsf{CC}}

\newcommand{\card}{\hbox{card}}
\newcommand{\Card}{\hbox{Card}}
\newcommand{\hcard}{\hbox{hcard}}
\newcommand{\Reg}{\hbox{Reg}}
\newcommand{\GCH}{\mathsf{GCH}}
\newcommand{\CH}{\mathsf{CH}}
\newcommand{\Ord}{\hbox{Ord}}
\newcommand{\Range}{\hbox{Range}}
\newcommand{\Dom}{\hbox{Dom}}
\newcommand{\Ult}{\hbox{Ult}}
\newcommand{\I}{\hbox{I}}
\newcommand{\II}{\hbox{II}}

\newcommand{\mouseM}{\mathcal{M}}
\newcommand{\mouseN}{\mathcal{N}}
\newcommand{\mouseP}{\mathcal{P}}
\newcommand{\treeT}{\mathcal{T}}

\newcommand{\forceQ}{\mathbb{Q}}
\newcommand{\forceP}{\mathbb{P}}
\newcommand{\forceR}{\mathbb{R}}
\newcommand{\seal}{\mathbb{S}(\vec{S})}

\newcommand{\Stat}{\mathcal{S}}
\newcommand{\NSA}{\hbox{NS}_{\omega_1} \upharpoonright A} 
\newcommand{\NS}{\hbox{NS}_{\omega_1}}
\newcommand{\NSK}{\hbox{NS}_{\kappa}}
\newcommand{\K}{K_{\omega_1}}

\newcommand{\secret}[1]{}
\secret{Versteckter Text}

\begin{abstract}
We show that under the assumption of the existence of the canonical inner model with one Woodin cardinal $M_1$, there
is a model of $\ZFC$ in which $\NS$ is $\aleph_2$-saturated
and $\Delta_1$-definable with $\omega_1$ as a 
parameter which answers a question of 
Sy-David Friedman and Liuzhen Wu. We also show that starting from an arbitrary universe with a Woodin cardinal, there is a model with $\NS$ saturated and
$\Delta_1$-definable with a ladder system $\vec{C}$ and a full Suslin tree
$T$ as parameters. Both results rely on a new coding technique whose presentation is the main goal of this article.
\end{abstract}

\section{Preliminaries}

\subsection{Introduction}
The investigation of the nonstationary ideal on $\omega_1$ and its saturation has a long history in
set theory. Recall that $\NS$ being ($\aleph_2$-)saturated means that $P(\omega_1) \slash \NS$ 
seen as a Boolean algebra has the $\aleph_2$-cc.
That $\omega_1$ can carry a normal, $\sigma$-complete and saturated ideal was already noted by K. Kunen in 1970. He obtained the result
assuming the existence of a huge cardinal. Using completely different methods which served as the starting point for the later development of Woodin's $\forceP_{max}$, J. Steel and R. Van Wesep obtained a decade later that already the nonstationary ideal on $\omega_1$ can be saturated. Considering the problem from a very different perspective again, Foreman, Magidor and Shelah showed a couple of years later that Martin's Maximum $\MM$, whose consistency can be derived from a supercompact cardinal, implies outright that $\NS$ is saturated.
Eventually Shelah found that already a Woodin cardinal is sufficient
to force a model in which $\NS$ is saturated, an assumption which turned out to be sharp in terms of consistency strength, as shown in 2006 by R. Jensen and Steel.

There are several interesting and deep interactions, connecting assertions related to "$\NS$ is saturated" with definability properties of certain important families of sets in the $H(\omega_2)$ of the surrounding universe. As a paradigmatic example we mention Woodin's famous result that in the presence of a measurable cardinal, $\NS$ being saturated implies the definable failure of the continuum hypothesis.

Goal of this paper is the proof of the following theorem:

\begin{Theorem*}
 Assume the existence of $M_1$, then there is a model of $\ZFC$ in which
 $\NS$ is $\aleph_2$-saturated and $\Delta_1$-definable with parameter
 $\omega_1$.
\end{Theorem*}

In fact the coding methods we will introduce will work over an arbitrary ground model with a Woodin cardinal once we allow more parameters. So the above theorem is an application of the more general theorem we shall prove in this paper.
\begin{Theorem*}
Assume that $V$ is a universe with a Woodin cardinal, $\vec{C}$ is a ladder system on $\omega_1$ and $T$ is a full Suslin tree. Then there is a generic extension $V[G]$ of $V$ such that in $V[G]$, the nonstationary ideal $\NS$ is $\aleph_2$-saturated and $\Delta_1$-definable over $H(\omega_2)$ with parameters $\vec{C}$ and $T$.

\end{Theorem*}

Starting point for this work
was the following remarkable theorem of H. Woodin (\cite{W}) who showed that
from $\omega$-many Woodin cardinals one can get a model in which
$\NS$ is $\omega_1$-dense, which in particular implies
the following:
\begin{Theorem*}
 $Con(\ZFC + $ ``there are $\omega$-many Woodin cardinals''$)$ implies \\$Con(\ZFC$ $+$ ``$\NS$ is both
 $\aleph_2$-saturated and $\Delta_1$-definable with parameters in $H(\omega_2)$''$)$.
\end{Theorem*}

S. D. Friedman and Liuzhen Wu in their \cite{FL} asked whether the 
assumption of $\omega$-many Woodin cardinals can be replaced by
a milder large cardinal axiom. 
Partial progress was made in \cite{FH}, where it is proved that from the existence of $M_1$ one can construct
a model in which $\NS$ is $\Delta_1$-definable with parameter $K_{\omega_1}$
while for a previously fixed stationary, co-stationary $A \subset \omega_1$, the
restricted nonstationary ideal $\NSA$ is $\aleph_2$-saturated, however the methods used there
did not give an answer for the full nonstationary ideal.

This paper presents a new approach to the problem, using a new coding technique
and a different set up of the proof, in order to yield the desired result. Its main idea is to use robust coding forcings to first generically create a suitable ground model over which a second coding forcing is applied to yield the desired result. This is in contrast to the traditional approaches, whose codings usually take place over some core model, whose definition comes along with a certain degree of generic robustness, which enables the possibility of finding the information created with the coding forcings.

Put in greater context this work can be seen as a new 
instance of the general quest of set theory (see e.g. \cite{CF}, \cite{FF}, \cite{FL})
which aims for the construction of models with interesting
features, usually obtained by assuming the existence of a large 
cardinal in the ground model, and additionally allow some
robust description of some of its most important families of sets.
The investigation of such problems has a long history in
set theory. As mentioned already, the usual procedure for obtaining such results (as in \cite{CF} or \cite{FF}) is starting with a suitable inner model whose definition is absolute for generic extensions (e.g. G\"odel's $L$ or some core model below a (not too) large cardinal) and apply some iterated forcing constructions to encode the desired information. This method has its limitations in the presence of more complicated inner models due to the lack of a sufficient amount of generic absoluteness and condensation of the inner model. This work is an attempt to circumvent these difficulties.

The question of $\NS$ being $\Delta_1$-definable
came to prominence after the introduction of 
the canary tree by A.H. Mekler and S. Shelah. In \cite{MekShe}
they proved that consistently $\NS$ is $\Delta_1$-definable with parameters in $H(\omega_2)$\footnote{The proof of \cite{MekShe} has a flaw which was found and corrected by T. Hyttinen and M. Rautila, see \cite{HytRau}.}. 
On the other hand if $V=L$ and $\kappa > \omega_1$, $\NSK$ can not be $\Delta_1$-definable with parameters in $H(\kappa^+)$ (see \cite{Generalized Descriptive Set Theory} Theorem 49.3).
In \cite{FLZ} it is proved that, starting from $L$ as the ground model and with $\kappa$ a successor cardinal, there is a cardinal and $\GCH$-preserving forcing
notion $\forceP$ such that in $L^{\forceP}$, $\NSK$ is $\Delta_1$-definable over $H(\kappa^+)$ with $H(\kappa)$ as the parameter.
In the context of large cardinals it is proved in \cite{FL}
that given a measurable cardinal there is a model in which
$\NS$ is precipitous and $\Delta_1$-definable with parameters from $H(\omega_2)$.
It is also observed there that under the assumption of
the existence of $P(\omega_1)^{\#}$ and ``$\NS$ is saturated'',
$\NS$ can not be $\Delta_1$-definable with parameter $\omega_1$. The argument utilizes Woodin's $\forceP_{max}$-forcing.
In the light of the last result, the theorem of this paper is somewhat  optimal. 

There are more and very recent results which show that under strong assumptions, the nonstationary ideal $\NS$ can not be $\Delta_1(\omega_1)$-definable.
First the result of Friedman and Wu has been extended by P. L\"ucke, R. Schindler and P. Schlicht in \cite{LueckeSchindlerSchlicht}, where they prove under the assumption of a Woodin cardinal and a measurable cardinal above, that no subfamily of $P(\omega_1)$ which is $\Delta_1(\omega_1)$-definable over $H(\omega_2)$ can separate the club filter on $\omega_1$ from the nonstationary ideal.
Secondly, a very recent and yet unpublished result from P. Larson, R. Schindler and L. Wu shows that if $\BMM$ holds and a Woodin cardinal exists, then $\NS$ can not be $\Delta_1(X)$ definable for any parameter $X \subset \omega_1$. Their argument uses Woodin's stationary tower forcing.
Consequentially a very interesting picture starts to emerge connecting the (impossibility of) $\Delta_1$-definability of $\NS$ with forcing axioms, large cardinals and their consequences.

We end the introduction noting that all the results previously obtained for $\NS$ being $\Delta_1(\omega_1)$ definable hold only in the presence of $\CH$, which prompted a question in \cite{FLZ} whether there are models of "$\NS$ is $\Delta_1(\omega_1)$-definable" and the failure of $\CH$. As in the models we will construct $2^{\aleph_0}=\aleph_2$ holds, we can answer it in the positive.

\section{Peliminaries}
\subsection{Some results on Suslin trees}
Suslin trees are one of the three coding techniques we will use during
the proof. 
Recall that a set theoretic tree $(T, <)$ is a Suslin tree if it is a normal tree of height $\omega_1$
and no uncountable antichain. All the trees which appear in this paper will be normal, thus whenever we talk about trees it is implicitly assumed that these trees are normal.
It is central for our needs to have a criterion which guarantees that
a Suslin tree $S$ will remain Suslin after passing to a generic extension of the universe. Recall that for a forcing $\forceP$ and $M \prec H(\theta)$, a condition $q \in \forceP$ is $(M,\forceP)$-generic iff for every maximal antichain $A \subset \forceP$, $A \in M$, it is true that $ A \cap M$ is predense below $q$.
The key fact is the following (see \cite{Miyamoto2} for the case where $\forceP$ is proper)
\begin{Lemma}\label{preservation of Suslin trees}
 Let $T$ be a Suslin tree, $S \subset \omega_1$ stationary and $\forceP$ an $S$-proper
 poset. Let $\theta$ be a sufficiently large cardinal.
 Then the following are equivalent:
 \begin{enumerate}
  \item $\Vdash_{\forceP} T$ is Suslin
 
  \item if $M \prec H_{\theta}$ is countable, $\eta = M \cap \omega_1 \in S$, and $\forceP$ and $T$ are in $M$,
  further if $p \in \forceP \cap M$, then there is a condition $q<p$ such that 
  for every condition $t \in T_{\eta}$, 
  $(q,t)$ is $(M, \forceP \times T)$-generic.
 \end{enumerate}

\end{Lemma}

\begin{proof}
For the direction from left to right note first that $\Vdash_{\forceP} T$ is Suslin implies $\Vdash_{\forceP} T$ is ccc, and in particular it is true that for any countable elementary submodel $N[\dot{G}_{\forceP}] \prec H(\theta)^{V[\dot{G}_{\forceP}]}$,  $\Vdash_{\forceP} \forall t \in T (t$ is $(N[\dot{G}_{\forceP}],T)$-generic). Now if $M \prec H(\theta)$ and $M \cap \omega_1 = \eta \in S$ and $\forceP,T \in M$ and $p \in \forceP \cap M$ then there is a $q<p$ such $q$ is $(M,\forceP)$-generic. So $q \Vdash \forall t \in T (t$ is $(M[\dot{G}_{\forceP}], T)$-generic, and this in particular implies that $(q,t)$ is $(M, \forceP \times T)$-generic for all $t \in T_{\eta}$. 

For the direction from right to left assume that $\Vdash \dot{A} \subset T$ is a maximal antichain. Let $B=\{(x,s) \in \forceP \times T \, : \, x \Vdash_{\forceP} \check{s} \in \dot{A} \}$, then $B$ is a predense subset in $\forceP \times T$. Let $\theta$ be a sufficiently large regular cardinal and let $M \prec H(\theta)$ be countable such that $M \cap \omega_1=\eta \in S$ and $\forceP, B,p,T \in M$. By our assumption there is a $q <_{\forceP} p$  such that $\forall t \in T_{\eta} ((q,t)$ is $(M, \forceP \times T)$-generic). So $B \cap M$ is predense below $(q,t)$ for every $t \in T_{\eta}$, which yields that $q \Vdash_{\forceP} \forall t \in T_{\eta} \exists s<_{T} t(s \in \dot{A})$ and hence $q \Vdash \dot{A} \subset T \upharpoonright \eta$, so $\Vdash_{\forceP} T$ is Suslin.
\end{proof}

As for iterations, T. Miyamoto (see \cite{Miyamoto}) defined a generalization of the usual iterations with revised countable support which he called nice iterations
which share the useful properties of iterations with revised countable support and additionally satisfy that whenever
the factors of a nicely supported iteration do not kill Suslin trees then the nice limit will preserve
Suslin trees as well. As nice iterations are quite technical and complicated, we refer the interested reader to the next section of this paper where we introduce the concept more thoroughly and prove the properties of nice iterations we use and which are not already proved in \cite{Miyamoto}. The reader should be able to follow everything below as long as she is willing to accept the usage of four facts about nice iterations.
\begin{Fact}
Let $((\forceP_{\alpha}, \dot{\forceQ}_{\alpha}) \, : \, \alpha \le \lambda)$ be
a nice iteration of length $\lambda \in$ Lim. Then

\begin{enumerate}
\item If $\lambda$ is an inaccessible cardinal and for every $\alpha < \lambda$, $\forceP_{\alpha}$ has the $\lambda$-cc, then $\forceP_{\lambda}$ is the direct limit of the $\forceP_{\alpha}$'s.
\item If $\lambda$ is Mahlo and every factor of the iteration has size less than $\lambda$, then the nice iteration has the $\lambda$-c.c.
\item If $\alpha < \lambda$ and if $\forceP_{\alpha \lambda}$ denotes the tail iteration of $((\forceP_{\alpha}, \dot{\forceQ}_{\alpha}) \, : \, \alpha \le \lambda)$, then $\Vdash_{\alpha} "\forceP_{\alpha \delta}$ is a nice iteration".
\item Let $S$ be a Suslin tree. If for all $\alpha$, $\Vdash_{\alpha}$ `` $\dot{\forceQ}_{\alpha}$ is semiproper
 and $S$ is a Suslin tree.'', then $\Vdash_{\lambda}$ ``$S$ is a Suslin tree.'' Also the $\lambda$-length iteration $\forceP_{\lambda}$ will be a semiproper forcing as well.
\end{enumerate}

\end{Fact}

Recall that for two trees $(T_0, <_{T_0})$ and $(T_1, <_{T_1})$ their tree-product $T_0 \times T_1$ is defined to be the tree which consists of nodes $\{ (t_0,t_1) \, : \, t_0 \in T_0 \land t_1 \in T_1 \land$ height$(t_0) =$ height$(t_1) \}$, ordered by $(t_0,t_1) <_{T_0 \times T_1} (s_0, s_1)$ if and only if $t_0 <_{T_0} s_0$ and $t_1 <_{T_1} s_1$. From now on whenever we talk about a product of trees, it is always the tree product which is meant.
For our purposes it is necessary to iteratively add 
sequences of blocks of Suslin trees $(\bar{T}_{\alpha} \, : \, \alpha < \kappa)$
such that $\bar{T}$ is itself an $\omega$-length sequence of Suslin trees whose finite subproducts are Suslin again.
One can construct such sequences using Jech's forcing which adds a
Suslin tree with countable conditions.

\begin{Definition}
 Let $\forceP_J$ be the forcing whose conditions are
 countable, normal trees ordered by end-extension, i.e. $T_1 < T_2$ if and only
 if $\exists \alpha < \text{height}(T_1) \, T_2= \{ t \upharpoonright \alpha \, : \, t \in T_1 \}$
\end{Definition}
It is wellknown that $\forceP_J$ is $\sigma$-closed and
adds a Suslin tree. In fact more is true, the generically added tree $T$ has 
the additional property that for any Suslin tree $S$ in the ground model
$S \times T$ will be a Suslin tree in $V[G]$.
\begin{Lemma}
 Let $V$ be a universe and let $S \in V$ be a Suslin tree. If $\forceP_J$ is 
 Jech's forcing for adding a Suslin tree and if $T$ is the generic tree
 then $$V[T] \models T \times S \text{ is Suslin.}$$
\end{Lemma}

\begin{proof}
Let $\dot{T}$ be the $\forceP_J$-name for the generic Suslin tree. We claim that $\forceP_J \ast \dot{T}$ has a dense subset which is $\sigma$-closed. As $\sigma$-closed forcings will always preserve ground model Suslin trees, this is sufficient. To see why the claim is true consider the following set:
$$\{ (p, \check{q}) \, : \, p \in \forceP_J \land height(p)= \alpha+1  \land  \check{q} \text{ is a node of $p$ of level } \alpha \}.$$
It is easy to check that this set is dense and $\sigma$-closed in $\forceP_J \ast \dot{T}$.

\end{proof}

A similar observation shows that a we can add an $\omega$-sequence of
such Suslin trees with a fully supported iteration. Even longer sequences
of such trees are possible if we lengthen the iteration but for our needs $\omega$-blocks are sufficient.

\begin{Lemma}\label{ManySuslinTrees}
 Let $S$ be a Suslin tree in $V$ and let $\forceP$ be a fully supported
 iteration of length $\omega$ of forcings $\forceP_J$. Then in the generic extension
 $V[G]$ there is an $\omega$-sequence of Suslin trees $\vec{T}=(T_n \, : \, n \in \omega)$ such
that for any finite $e \subset \omega$
the tree $S \times \prod_{i \in e} T_i$ will be a Suslin tree in $V[\vec{T}]$.
\end{Lemma}
\begin{proof}
Let $G$ be a generic filter for $\forceP$. First we observe that the fully supported product $\prod_{n \in \omega} \forceP_J$ followed by the fully supported forcing with the product of the generically added trees $T_n$ has a $\sigma$-closed dense subset which is defined coordinate-wise as above, thus ground model Suslin trees are preserved. If we pick a finite $e \subset \omega$, then for an arbitrary Suslin tree $S \in V$, $\prod_{i \in e} T_i \times S$ is a Suslin tree in the intermediate model generated by the trees $T_i$, $i \in e$ over $V$. This is preserved when passing to the generic extension $V[G]$.
\end{proof}

These sequences of Suslin trees will become important later in our proof, thus they will get a name.
\begin{Definition}
 Let $\vec{T} = (T_{\alpha} \, : \, \alpha < \kappa)$ be a sequence of Suslin trees. We say that the sequence is an 
 independent family of Suslin trees if for every finite set $e= \{e_0, e_1,...,e_n\} \subset \kappa$ the product $T_{e_0} \times T_{e_1} \times \cdot \cdot \cdot \times T_{e_n}$ 
 is a Suslin tree again.
\end{Definition}

\subsection{Nice Iterations}
This section contains a quick reminder of the main definitions and properties of T. Miyamoto's concept of nice iterations, which we use in our proof. We will provide proofs for the three properties of nice iterations which we took advantage of in the arguments before. Our notation will be parallel to \cite{Miyamoto}. In particular we will use his definition of an iteration (see \cite{Miyamoto}, Definition 1.6). Recall that for conditions $p$ in a forcing iteration, $l(p)$ denotes the length of $p$ as seen as a sequence, and $p \upharpoonright \alpha$ denotes the condition which is the sequence $p$ cut at $\alpha$.

The central concept to define what a nicely supported iteration is, is a nested antichain. 
\begin{Definition}
Suppose that $(\forceP_{\alpha} \, : \, \alpha < \nu)$ be an iteration of length $\nu$. A nested antichain in $(\forceP_{\alpha} \, : \, \alpha < \nu)$ is a triple $(T, (T_n \, : \, n \in \omega), (suc^n_{T} \, : \, n \in \omega ))$ such that
\begin{enumerate}
\item $T= \bigcup_{n \in \omega} \{ T_n \, : \, n \in \omega \}.$
\item $T_0 = \{a_0\}$ for some condition $a_0 \in \bigcup \{ \forceP_{\alpha} \,: \, \alpha < \nu \}$.
\item $T_n \subset \bigcup_{n \in \omega} \{ \forceP_{\alpha} \, : \, \alpha < \nu \}$ and $suc^n_{T} : T_n \rightarrow P(T_{n+1})$.
\item For $a \in T_n$ and $b \in suc^n_{T}(a)$, $l(a) \le l(b)$ and $b \upharpoonright l(a) \le a$.
\item For $a \in T_n$ and $b,b' \in suc^n_{T}(a)$, $b \ne b'$ iff $b \upharpoonright l(a)$ and $b' \upharpoonright l(a)$ are incompatible in $\forceP_{l(a)}$.
\item For $a \in T_n$, $\{b \upharpoonright l(a) \, : \, b \in suc^n_{T}(a) \}$ is a maximal antichain below $a$ in $\forceP_{l(a)}$. 
\item $T_{n+1} = \bigcup \{ suc^n_{T} (a) \, : \,  a \in T_n \}$.
\end{enumerate}

\end{Definition}
From now on we will simply write $T$ for a nested antichain, and suppress the mentioning of the levels $T_n$ and the successor function $suc^n_{T}$.
A nice limit $\forceP_{\nu}$ of an iteration $(\forceP_{\alpha} \, : \, \alpha < \nu\}$ will consist of conditions which correspond to nested antichains in $(\forceP_{\alpha} \, : \, \alpha < \nu )$. To achieve that we need to write a nested antichain again as a sequence.

\begin{Definition}
Let $T$ be a nested antichain in an iteration $(\forceP_{\alpha} \, : \, \alpha < \nu)$. Let $\beta < \nu$ be arbitrary, then we say that $y \in \forceP_{
\beta}$ is a mixture of $T$ up to $\beta$ iff for all $i < \beta$, $y \upharpoonright i$ forces:
\begin{enumerate}
\item $y(i) = a_0(i)$, if $i < l(a_0)$ and $a_0 \upharpoonright i \in G_i$, where $T_0 = \{a_0\}$.
\item $y(i)= b(i)$, if there is a pair $(a,b)$ such that $a,b \in T$, $b \in suc(a)$, $l(a) \le i < l(b)$ and $b \upharpoonright i \in G_i$.
\item $y(i) = 1$, if there is a sequence $(a_n \, : \, n \in \omega)$ such that $a_0 \in T_0$ and  for all $n \in \omega$, $a_{n+1} \in suc^n_{T} (a_n)$, $l(a_n) \le i$ and $a_n \in G_i \upharpoonright l(a_n)$.
\item No requirements else.
\end{enumerate}

\end{Definition}

It is important to note that for a given nested antichain $T$ in $(\forceP_{\alpha} \, : \, \alpha < \nu)$, a mixture of $T$ up to some $\beta$ need not to exist. One can weaken the requirements in the following way.
\begin{Definition}
Let $(\forceP_{\alpha} \, : \, \alpha < \nu )$ be an iteration, assume that $T$ is a nested antichain in $(\forceP_{\alpha} \, : \, \alpha < \nu)$, and let $\beta \le \nu$ be a limit ordinal. Then a sequence $y$ of length $\beta$ is $(T, \beta)$-nice iff for all $\alpha < \beta$, $y \upharpoonright \alpha \in \forceP_{\alpha}$ and $y \upharpoonright \alpha$ is a mixture of $T$ up to $\alpha$. 

\end{Definition}
Note that in the definition above, we do not demand that $y$ is a condition in $\forceP_{\beta}$.
The intuition is that our iteration at limits should consist of sequences which correspond to nested antichains. This reasoning seems circular at first sight, the nested antichains need an iteration already to even define them. But a careful definition takes care of that. We define first conditions at limit stages using nested antichains of the old sequence of posets, and then show that in the new limit, the new nested antichains for that new poset are already represented as conditions automatically.

With these notions it is possible to define what a nice limit of an iteration is. 
\begin{Definition}
Assume that $(\forceP_{\alpha} \, : \, \alpha < \nu)$ is an iteration, $\nu$ a limit ordinal. We define a separative preorder, the nice limit $(\forceP_{\nu}, 1_{\nu}, <)$ of $(\forceP_{\alpha}\, : \, \alpha < \nu)$ as follows:
\begin{enumerate}
\item $\forceP_{\nu}:= \{ x \, : \, x$ is a sequence of length $\nu$ and there is a nested antichain $T$ in $(\forceP_{\alpha} \, : \, \alpha < \nu)$ such that $x$ is $(T,\nu)$-nice $\}$.
\item If $x,y \in \forceP_{\nu}$ then $x \le_{\nu} y$ if and only if for all $\alpha < \nu$, $x \upharpoonright \alpha \le_{\alpha} y \upharpoonright \alpha$.
\item Maximal element $1_{\nu}$ is defined to be the $\nu$-sequence $\bigcup_{\alpha < \nu} 1_{\alpha}$.
\end{enumerate}
\end{Definition}

One can show the following fundamental properties:
\begin{thm}(2.9. pp.1445 \cite{Miyamoto})
Let $(\forceP_{\alpha} \, : \, \alpha < \nu )$ be an iteration and $\forceP_{\nu}$ its nice limit. Then $(\forceP_{\alpha} \, : \, \alpha \le \nu)$ is an iteration and for every nested antichain $T$ in $(\forceP_{\alpha} \, : \, \alpha \le \nu )$, and any $\nu$-length sequence $x$ which is $(T, \nu)$-nice we have that $x \in \forceP_{\nu}$, thus $x$ is a mixture of $T$ up to $\nu$.
\end{thm}
A nice iteration is defined as follows (see \cite{Miyamoto}, Definition 3.6)
\begin{Definition}
An iteration $(\forceP_{\alpha} \, : \, \alpha < \nu )$ is nice iff
\begin{enumerate}
\item For any $i+1<\nu$, if $p \in \forceP_i$ and $\tau$ is a $\forceP_i$-name such that $p \Vdash_i "\tau \in \forceP_{i+1} \land \tau \upharpoonright i \in \dot{G}_i"$, then there is a $q \in \forceP_{i+1}$ such that $q \upharpoonright i =p$ and $p \Vdash_i " \tau (i) \equiv q (i)"$
\item For any limit ordinal $\beta < \nu$ and any sequence $x$ of length $\beta$, $x \in \forceP_{\beta}$ holds iff there is a nested antichain $T$ in $(\forceP_{\alpha} \, : \, \alpha < \nu )$ such that $x$ is $(T, \beta)$-nice.
\end{enumerate}
\end{Definition}
One can produce nice iterations of length $\delta$ in the usual way: we form a sequence $(\forceP_{\alpha}, \dot{\forceQ}_{\alpha} \, : \, \alpha < \delta)$, where for every $\alpha < \delta$, $\dot{\forceQ}_{\alpha}$ is a $\forceP_{\alpha}$-name for a separative partial order in $V^{\forceP_{\alpha}}$, and $\forceP_{\alpha+1}$ is isomorphic to $\forceP_{\alpha} \ast \dot{\forceQ}_{\alpha+1}$, and $\forceP_{\beta}$ is the nice limit of the $\forceP_{\alpha}$'s for $\alpha<\beta$ and limit ordinals $\beta$.

As usual, it is possible to cut an iteration $(\forceP_{\alpha} \, : \, \alpha < \delta)$ at some intermediate stage, and look at the tail iteration $\forceP_{\alpha \delta}$ as seen from the new ground model $V^{\forceP_{\alpha}}$. If $G_{\alpha}$ denotes a $\forceP_{\alpha}$-generic filter, then $\forceP_{\alpha \delta} = \{ p \upharpoonright [\alpha, \delta) \, : \, p \in \forceP_{\delta} \land p \upharpoonright \alpha \in G_{\alpha}\}$. As conditions in nice iterations correspond to nested antichains and nested antichains can be cut in a canonical way (see Lemma 2.7 in \cite{Miyamoto}), we obtain that tail iterations of nice iterations are nice iterations as seen from the intermediate stage where the cut happend.
\begin{thm}
For a nice iteration $(\forceP_{\alpha} \, : \, \alpha < \delta)$, if $\alpha < \delta$ then $\Vdash_{\alpha} "\forceP_{\alpha \delta}$ is a nice iteration$"$
\end{thm}
\begin{proof}
We shall show that for an arbitrary $\alpha< \delta$ and any limit $\beta \le \delta$, $\alpha < \beta$, $\Vdash_{\alpha} x \in \forceP_{\alpha \beta} \Leftrightarrow \exists$ nested antichain $\dot{T}$ in $(\forceP_{\alpha \gamma} \, : \, \gamma < \beta)$ such that $x$ is $(\dot{T}, \alpha \beta)$-nice. 

For the forward direction, we note that if $x \in \forceP_{\alpha \beta}$ then there is a $p \in \forceP_{\beta}$ such that $p \upharpoonright \alpha \Vdash p\upharpoonright [\alpha, \beta)=x$. As $\forceP_{\beta}$ is a nice iteration, there is a nested antichain $T$ such that $p$ is $(T, \beta)$-nice. As a result $p \upharpoonright \alpha$ is a mixture of $T$ up to $\alpha$. Then there exists a $\forceP_{\alpha}$-name $\dot{T}$ such that 
\[ p \upharpoonright \alpha \Vdash_{\alpha} \dot{T} \text{ is a nested antichain in } (\forceP_{\alpha \gamma} \, : \, \gamma < \beta )\] and
\[p \upharpoonright \alpha \Vdash_{\alpha} p \upharpoonright [\alpha,\beta) (=x) \text{ is } (\dot{T},\alpha \beta)\text{-nice}\]
which is what we wanted.

For the backward direction, let $p \in \forceP_{\alpha}$ be such that 
\begin{align*}
p \Vdash_{\alpha} x \text{ is a sequence of length } (\beta -\alpha) \land \exists \dot{T} (\dot{T}\text{ is a nested antichain in } \\ (\forceP_{\alpha \gamma})_{\gamma < \beta} \text{ such that } x \text{ is $(\dot{T},\alpha \beta)$-nice).}
\end{align*}
We shall show that $p \Vdash_{\alpha} x \in \forceP_{\alpha \beta}$.
It is straightforward to turn the $\forceP_{\alpha}$-name $\dot{T}$ for a nested antichain in $(\forceP_{\alpha, \gamma})_{\gamma < \beta}$ into a nested antichain $T$ in $(\forceP_{\gamma})_{\gamma < \beta}$ such that $a_0$ of $T$ is $p$. As $\forceP_{\beta}$ is a nice iteration, there is a condition $p' \in \forceP_{\beta}$ such that $p'$ is a $T$-mixture up to $\beta$. As $a_0$ of $T$ is $p $, we know that $p' \upharpoonright \alpha = p$. Thus we arrive at a condition $p' \in \forceP_{\beta}$ such that $p' \upharpoonright \alpha \Vdash p' \upharpoonright [\alpha, \beta)$ is a $\dot{T}$-mixture up to $\beta$, and by assumption $p' \upharpoonright \alpha \Vdash  x$ is $(\dot{T}, \alpha \beta)$-nice, and as $\dot{T}$ uniquely determines the condition due to \cite{Miyamoto}, Cor. 2.6, we infer that $p' \upharpoonright \alpha \Vdash x= p' \upharpoonright [\alpha, \beta)$, hence $p  \Vdash x \in \forceP_{\alpha \beta}$ as desired.

\end{proof}

We shall show that if $\kappa$ is a Mahlo cardinal and $(\forceP_{\alpha} \, : \, \alpha < \kappa)$ is a nicely supported iteration of length $\kappa$ and every $\forceP_{\alpha}$ has the $\kappa$-c.c. then the limit $\forceP_{\kappa}$ has the $\kappa$-c.c. as well.
To see this, note first that for an inaccessible $\delta$ and a nice iteration $(\forceP_{\alpha} \, : \, \alpha < \delta)$ such that every $\forceP_{\alpha}$ has the $\delta$-c.c. the nested antichains, and hence all the conditions in the nice iteration are bounded below $\delta$ in terms of their support. It follows that the nice limit of $(\forceP_{\alpha} \, : \, \alpha < \delta)$ has size at most $\delta$, and thus the $\delta^{+}$-chain condition and is the same as just taking the direct limit at such stages. Thus if $\kappa$ is a Mahlo cardinal, then the set of inaccessibles below $\kappa$ is stationary, thus we have a stationary set of stages for which we take the direct limit of our iteration and by a well-known theorem (see e.g. Theorem 16.30 in \cite{Jech}) the nice limit will satisfy the $\kappa$-c.c.
As a summary
\begin{Lemma}
If $\kappa$ is a Mahlo cardinal and $(\forceP_{\alpha} \, : \, \alpha < \kappa)$ is a nice iteration of factors which have size less than $\kappa$ then the nice limit will satisfy the $\kappa$-c.c.

\end{Lemma}

\subsection{Coding reals by triples of ordinals}
We present a coding method invented by A. Caicedo and B. Velickovic (see \cite{CV})  which we will use in the argument.

\begin{Definition}
A $\vec{C}$-sequence, or a ladder system, is a sequence $(C_{\alpha} \, : \, \alpha \in \omega_1, \alpha \text{ a limit ordinal })$, such
that for every $\alpha$, $C_{\alpha} \subset \alpha$ is cofinal and the ordertype of $C_{\alpha}$ is $\omega$.
\end{Definition}

For three subsets $x,y,z \subset \omega$ we can define an oscillation function. First turn the set $x$ 
into an equivalence relation $\sim_x$, defined on the set $\omega- x$ as follows: for natural numbers in the 
complement of $x$ satisfying $n \le m$ let $n \sim_x m$ if and only if $[n,m] \cap x = \emptyset$.
This enables us to define:
\begin{Definition}
For a triple of subset of natural numbers $(x,y,z)$ list the intervals $(I_n \, :\, n \in k \le \omega)$ of equivalence classes of
$\sim_x$ which have nonempty intersection with both $y$ and $z$. Then the oscillation map $o(x,y,z):
k \rightarrow 2$ is defined to be the function satisfying

\begin{equation*}
o(x,y,z)(n) = \begin{cases}
0  & \text{ if min}(I_n \cap y) \le \text{min}(I_n \cap z) \\ 1 & \text{ else}
\end{cases}
\end{equation*}

\end{Definition}

Next we want to define how suitable countable subsets of ordinals can be used to code reals. First we fix a ladder system $\vec{C}$ for the rest of this section.
Suppose that $\omega_1 < \beta < \gamma < \delta$ are fixed limit ordinals of uncountable cofinality, and that 
$N \subset M$ are countable subsets of $\delta$.
Assume further that $\{ \omega_1, \beta, \gamma\} \subset N$ and that for every 
$\eta \in \{ \omega_1, \beta, \gamma\}$, $M \cap \eta$ is a limit ordinal and $N \cap \eta < M \cap \eta$.
We can use $(N,M)$ to code a finite binary string. Namely let $\bar{M}$ denote the transitive collapse of 
$M$, let $\pi : M \rightarrow \bar{M}$ be the collapsing map and let 
$\alpha_M := \pi(\omega_1)$, $\beta_M := \pi(\beta), \, \gamma_M := \pi(\gamma) \, \delta_M:= \bar{M}$. 
These are all countable limit ordinals.
Further set $\alpha_N:= sup(\pi``(\omega_1 \cap N))$ and let the height $n(N,M)$ of $\alpha_N$ in $\alpha_M$ 
be the natural number defined by

$$n(N,M):= card (\alpha_N \cap C_{\alpha_M})$$ where $C_{\alpha_M}$ is an element of our previously fixed ladder system. 
As $n(N,M)$ will appear quite often in the following we write shortly $n$ for $n(N,M)$. Note that
as the ordertype of each $C_{\alpha}$ is $\omega$, and as $N \cap \omega_1$ is bounded below $M \cap \omega_1$,
$n(N,M)$ is indeed a natural number.
Now we can assign to the pair $(N,M)$ a triple $(x,y,z)$ of finite subsets of natural numbers as follows:
$$x:= \{ card(\pi(\xi) \cap C_{\beta_M}) \, : \, \xi \in \beta \cap N \}.$$ Note that $x$ again is finite as $\beta \cap N$ is 
bounded in the cofinal in $\beta_M$-set $C_{\beta_M}$, which has ordertype $\omega$. Similarly we define 
$$y:= \{ card(\pi(\xi) \cap C_{\gamma_M}) \, : \, \xi \in \gamma \cap N \}$$ and
$$z:= \{ card(\pi(\xi) \cap C_{\delta_M} \, : \, \xi \in \delta \cap N \}.$$ Again it is easily 
seen that these sets are finite subsets of the natural numbers.
We can look at the oscillation $o(x \backslash n, y \backslash n, z \backslash n)$ (remember we let $n:= n(N,M)$)
and if the oscillation function at these points has a domain bigger or equal to $n$ then we write
\begin{equation*}
s_{\beta, \gamma, \delta} (N,M):= \begin{cases}
o(x \backslash n, y \backslash n, z \backslash n)\upharpoonright n & \text{ if defined } \\  \ast \text{ else}
\end{cases}
\end{equation*}
Where $\ast$ should simply be some error symbol. Similarly we let $s_{\beta, \gamma, \delta} (N,M) \upharpoonright l = \ast$ when $l > n$.
Finally we are able to define what it means for a triple of ordinals $(\beta, \gamma, \delta)$ to code a real $r$.

\begin{Definition}
For a triple of limit ordinals $\omega_1 < \beta < \gamma < \delta$ of uncountable cofinality, we say that it codes a real $r \in 2^{\omega}$
if there is a continuous increasing sequence $(N_{\xi} \, : \, \xi < \omega_1)$ of countable sets of ordinals whose 
union is $\delta$ and which satisfies that there is a club $C \subset \omega_1$
such that whenever $\xi \in C$ is a limit ordinal then there is a $\nu < \xi$ such that
$$ r = \bigcup_{\nu < \eta < \xi} s_{\beta, \gamma, \delta} (N_{\eta}, N_{\xi}). $$
We say that the sequence $(N_{\xi} \, : \, \xi < \omega_1)$ is a reflecting sequence.
\end{Definition}

Witnesses to the coding can be added with a proper forcing. 
On the other hand there is a certain amount of control
for fixed triples of ordinals and the behavior of 
continous, increasing sequences on them:

\begin{thm}[Caicedo-Velickovic]

\begin{enumerate}

\item[$(\dagger)$] Given ordinals $\omega_1 < \beta < \gamma < \delta < \omega_2$ of cofinality $\omega_1$,
there exists a proper notion of forcing $\forceP_{\beta \gamma \delta}$ such that after forcing with it the following holds:
There is an increasing continuous sequence $(N_{\xi} \, : \, \xi < \omega_1)$ such that $N_{\xi} \in [\delta]^{\omega}$ 
whose union is $\delta$ such that for every limit $\xi < \omega_1$ and every $n \in \omega$ there 
is $\nu < \xi$ and $s_{\xi}^n \in 2^n$ such that 
$$s_{\beta \gamma \delta}(N_{\eta}, N_{\xi}) \upharpoonright n = s_{\xi}^{n}$$ holds for every $\eta$ in the interval $(\nu, \xi)$.
We say then that the triple $(\beta, \gamma, \delta)$ is stabilized.
 
\item[$(\ddagger)$] Further if we fix a real $r$
there is a proper notion of forcing $\forceP_r$ such that the forcing will produce for a triple of ordinals 
$(\beta_r, \gamma_r, \delta_r)$ of size and cofinality $\omega_1$ a reflecting sequence 
$(P_{\xi} \, : \, \xi < \omega_1)$, $P_{\xi} \in [\delta_r]^{\omega}$ such that $\bigcup P_{\xi} = \delta_r$ 
and such that there is a club $C \subset \omega_1$ and for every limit $\xi \in C$ there is a $\nu < \xi$ such that 
$$\bigcup_{\nu < \eta < \xi} s_{\beta_r \gamma_r \delta_r} (P_{\eta}, P_{\xi}) = r.$$ 
\end{enumerate}
\end{thm}

Both partial orders $\forceP_{\beta \gamma \delta}$ and $\forceP_{r}$ which force $(\dagger)$ and $(\ddagger)$ respectively are actually instances of a general class of notions of forcing which were investigated first by J. Moore's in his work on the Set Mapping Reflection Principle ($\MRP$) (see \cite{Moore}). We shall see later that these forcings never kill Suslin trees. For that reason we have to introduce a couple of notions from \cite{Moore}.
We need first the following local version of stationarity: 
\begin{Definition}
Let $\theta$ be a regular cardinal, $X$ be an uncountable set, let $M \prec H_{\theta}$ be a countable 
elementary submodel which contains $[X]^{\omega}$ as an element. Then $S \subset [X]^{\omega}$ is $M$-stationary if 
for every club subset $C$ of $[X]^{\omega}$, $C \in M$ it holds that $$C \cap S \cap M \ne \emptyset.$$
\end{Definition}

\begin{Definition}
Let $X$ be an uncountable set, $N \in [X]^{\omega}$ and $x \subset N$ finite. Then the Ellentuck 
topology on the set $[X]^{\omega}$ is generated by base sets of the form
$$ [x,N]:= \{ Y \in [X]^{\omega} \, : \, x \subset Y \subset N \}.$$ From now on whenever we say 
open we mean open with respect to the Ellentuck topology.
\end{Definition}

\begin{Definition}
Let $X$ be an uncountable set, let $\theta$ be a large enough regular cardinal so that
 $[X]^{\omega} \in H_{\theta}$. Then a function $\Sigma$ is said to be open stationary 
if and only if its domain is a club $C \subset [H_{\theta}]^{\omega}$ and for every 
countable $M \in C$, $\Sigma(M) \subset [X]^{\omega}$ is open and $M$-stationary.
\end{Definition}

Moore has shown that for any open stationary map $\Sigma$ it is possible to force a  reflecting sequence $(N_{\xi} \, : \, \xi < \omega_1)$ with a proper forcing $\forceP_{\Sigma}$ (see \cite{Moore}, Theorem 3.1.).

\begin{Proposition}[Moore]
Let $\Sigma$ be an open stationary function defined on some club $C \subset [H_{\theta}]^{\omega}$ with range $P([X]^{\omega})$ for some uncountable set $X$. Then there is a proper notion of forcing $\forceP_{\Sigma}$ which adds a continuous sequence of models $(N_{\xi} \, : \, \xi < \omega_1)$ (a reflecting sequence) in $dom(\Sigma)$ 
such that for every limit ordinal $\xi$ there is a $\nu < \xi$ such that for every $\eta$ with 
$\nu < \eta < \xi$, $N_{\eta} \cap X \in \Sigma(N_{\xi})$.
\end{Proposition}
The forcing $\forceP_{\Sigma}$ is defined as expected: for an open stationary map $\Sigma$ let $\forceP_{\Sigma}$ consist of conditions $p$ which are  functions $p: \alpha +1 \rightarrow dom(\Sigma)$, $\alpha$ countable, 
which are continuous and $\in$-increasing, and which additionally satisfy the 
$\MRP$-condition on its limit points, namely that for every $0< \nu < \alpha$, $\nu$ a limit ordinal,  
there is a $\nu_0 < \nu$ such that $p(\xi) \cap X \in \Sigma(p(\nu))$ for every 
$\xi$ in the interval $(\nu_0, \nu)$. The order is by extension.

Now, as already mentioned above, both forcings $\forceP_{\beta \gamma \delta}$ and $\forceP_{r}$ which will produce $(\dagger)$ and $(\ddagger)$ respectively are of the form $\forceP_{\Sigma}$ (see \cite{CV}, Lemma 1, 4 and 5).

\begin{Proposition}\label{MRPpreservesSuslin}
There are two open stationary maps $\Sigma_r$ and $\Sigma_{\beta \gamma \delta}$ such that $\forceP_{\beta \gamma \delta} = \forceP_{\Sigma_{\beta \gamma \delta}}$ and $\forceP_r = \forceP_{\Sigma_r}$.
\end{Proposition}
As a consequence, if we show that $\forceP_{\Sigma}$ always preserves Suslin trees for $\Sigma$ an arbitrary open stationary map, we will have proven that $\forceP_r$ and $\forceP_{\beta \gamma \delta}$ preserve Suslin trees. This is indeed the case as we will show now.

\begin{Proposition}\label{PFAMRP}
Let $\Sigma$ be an arbitrary open stationary map and let $\forceP_{\Sigma}$ be as defined above. Then $\forceP_{\Sigma}$ preserves Suslin trees.
\end{Proposition}
\begin{proof}
 We fix an arbitrary Suslin tree $T$. For every countable $M \prec H_{\lambda}$, if $\eta = M \cap \omega_1$ and $t \in T_{\eta}$ then by the Susliness of $T$, $t$ is $(M,T)$-generic. As a consequence we can consider the generic extension $M[t]$ and note that $M[t]$ will not add any new countable sets of ordinals to $M$. 
 
Let $\Sigma$ be the open stationary map. Recall that $\forceP_{\Sigma}$ was defined to consist of conditions $p : \alpha + 1 \rightarrow dom(\Sigma)$, $\alpha < \omega_1$, which are $\in$-increasing and continuous and which additionally satisfy that for every limit ordinal $\nu$, $0< \nu < \alpha$ 
there is a $\nu_0 < \nu$ such that $p(\xi) \cap X \in \Sigma(p(\nu))$ for every 
$\xi$ in the interval $(\nu_0, \nu)$. The order is by extension.

We let $\lambda$ be sufficiently large and pick a countable $M \prec H_{\lambda}$ which contains $T, \Sigma$, $\forceP_{\Sigma}$, a 
condition $p \in \forceP_{\Sigma}$, and the structure $H_{|\forceP_{\Sigma}|^+}$. 
Letting $M \cap \omega_1 =\eta$, then our goal is to produce a stronger condition $q< p$ such that for every $t \in T_{\eta}$, $(q,t)$ is an $(M, \forceP_{\Sigma} \times T)$-generic condition.

We list $(t_n \, : \, n \in \omega)$, the elements of $T_{\eta}$ and build the according generic extensions $M[t_n]$. As $M$ and $M[t_n]$ have the same countable sequences of $M$-elements, there is no difference when talking about $\Sigma$ and $\forceP_{\Sigma}$ in either $M$ or $M[t_n]$. 
We list all the dense subsets $(D_0, D_1,..)$ of $\forceP_{\Sigma}$
which we can find in $\bigcup_{n \in \omega} M[t_n]$ and build by recursion a descending sequence of conditions 
$(p_i \, : \, i \in \omega)$ in $M$, starting at $p_0:= p$ hitting the corresponding $D_{i-1}$. 

Assume that we have already built conditions up to $i \in \omega$. 
We consider the collection of models $N'_i$, where each $N'_i$ is a countable elementary submodel of $H_{|\forceP_{\Sigma}|^+}$ 
containing $H_{\theta}$, $D_i$, $\forceP_{\Sigma}$ and $p_i$, and 
build the club of countable structures $C_i:= \{ N'_i \cap X \, : \, N'_i \text{ as just described} \}$. 
Note that this club will be in $M$, hence $M \cap H_{\theta} \in C_i$.
Thus the set $M\cap H_{\theta}$ will be in
the domain of $\Sigma$ and by the definition of 
$\Sigma$, the set $\Sigma(M \cap H_{\theta})$ is $M\cap H_{\theta}$-stationary and open. 
So there is an $N_i=N'_i \cap X  \in C_i \cap \Sigma(M \cap H_{\theta}) \cap M$, and by the definition of the Ellentuck topology, 
there is a finite subset of $N_i$ called $x_i$ such that $[x_i, N_i] \subset \Sigma(M \cap H_{\theta})$.
We first extend the condition $p_i$ to $q_i := p_i \cup \{( \zeta_i +1, hull^{H_{\theta}} (p_i(\zeta_i) \cup x_i))\}$, 
for $\zeta_i$ the maximum of the domain of $p_i$. This condition $q_i$ will also be in $N'_i$ as all its defining 
parameters are, thus as $N'_i$ also contains $D_i$ we can extend the condition $q_i$ to a $p_{i+1} \in N'_i \cap D_i$. 
Note that as we are working in $N'_i$, no matter how we extend $q_i$, the range of the extended condition intersected with $X$ will always
be contained in $N_i=N'_i \cap X$, and as $\Sigma(M \cap H_{\theta}) \supset [x_i, N_i]$, 
it will also be contained in $\Sigma(M \cap H_{\theta})$. 
Then if we set $q= p_{\omega} := \bigcup_{i \in \omega} p_{i} \cup (\omega, (M\cap H_{\theta}))$ 
then this will be a condition in $\forceP_{\Sigma}$ as it forms an $\in$-increasing, continuous function from $\omega+1$ to $dom(\Sigma)$, for which $\forall n \in \omega (q(n) \cap X \in \Sigma(q(\omega))=M\cap H_{\theta})$ is true. By construction $q$ is 
below $p$ and $(M,\forceP_{\Sigma} \times T)$-generic, thus the forcing preserves the arbitrary Suslin tree $T$.
 
\end{proof}

\subsection{Almost disjoint coding}

The following subsection quickly reintroduces the almost disjoint coding forcing due to R. Jensen and R. Solovay \cite{JensenSolovay}. We will identify subsets of $\omega$ with their characteristic function and will use the word reals for elements of $2^{\omega}$ and subsets of $\omega$ respectively.
Let $F=\{f_{\alpha} \, : \, \alpha < 2^{\aleph_0} \}$ be a family of almost disjoint subsets of $\omega$, i.e. a family such that if $r, s \in F$ then 
$r \cap s$ is finite. Let $X\subset  \kappa$ for $\kappa \le 2^{\aleph_0}$ be a set of ordinals. Then there 
is a ccc forcing, the almost disjoint coding $\mathbb{A}_F(X)$ which adds 
a new real $x$ which codes $X$ relative to the family $F$ in the following way
$$\alpha \in X \text{ if and only if } x \cap f_{\alpha} \text{ is finite.}$$
\begin{Definition}
 The almost disjoint coding $\mathbb{A}_F(X)$ relative to an almost disjoint family $F$ consists of
 conditions $(r, R) \in \omega^{<\omega} \times F^{<\omega}$ and
 $(s,S) < (r,R)$ holds if and only if
 \begin{enumerate}
  \item $r \subset s$ and $R \subset S$.
  \item If $\alpha \in X$ then $r \cap f_{\alpha} = s \cap f_{\alpha}$.
 \end{enumerate}

\end{Definition}
There is another variant due to L. Harrington (see \cite{Harrington}) which codes sets of reals relative to
a new real.
For the following fix some definable bijection of finite sequences of integers
and $\omega$ and for $b \in \omega^{\omega}$ write $\bar{b} (n)$ for
the natural number which codes the finite sequence $b \cap n$.
A subset $b \subset \omega$ gives rise to a new set $S(b) \subset \omega$ if we
consider the set of the codes of its initial segments $\{ \bar{b}(n) \, : \, n \in \omega \}$.

\begin{Definition}
Suppose that $A \subset [\omega]^{\omega}$, then the almost disjoint coding forcing
 for $A$, $\mathbb{A}(A)$ is defined as follows. Conditions
 are pairs $(s(0), s(1))$ such that $s(0)$ is a finite set of natural numbers
 and $s(1)$ is a finite subset of the fixed set of reals $A$.
 For two conditions $r,s \in \mathbb{A}(A)$ we say
 $s<r$ if and only if
 \begin{itemize}
  \item $r(0) \subset s(0)$ and $r(1) \subset s(1)$
  \item $\forall a \in r(1)\, (S(a) \cap s(0) \subset r(0))$
 \end{itemize}

\end{Definition}
Note that $\mathbb{A}(A)$ has the Knaster property, thus products of $\mathbb{A}(A)$ have the ccc. Given a set of reals $A$ in the ground model $V$, the effect of forcing with $\mathbb{A}(A)$ is the following. The first coordinate of conditions in the generic filter $G$ will union up to a real $a$ which codes the set $A$ with the help of a predicate for $\omega^{\omega} \cap V$. In $V[G]$ membership in $A$ is characterized like this
$$x \in A \leftrightarrow x \in V \land S(x) \cap a \text{ is finite. }$$

This characterization will play an important role later. We will work towards a universe whose $H(\omega_2)$ will be definable in arbitrary generic extensions obtained using ccc forcings. Consequentially, we can use the just defined forcing to encode information into just one real, and the information can be correctly decoded in all further ccc extensions of the universe.

\subsection{$\NS$ saturated}

As our proof depends on Shelah's argument to force $\NS$ saturated from a Woodin cardinal we introduce very briefly
some of the main ideas. We later (see Theorem \ref{NS saturated}), assuming knowledege of \cite{SchindlerNS}, discuss in more detail how the proof can be altered. 

The crucial forcing notion which can be used to bound the length 
of antichains in $P(\omega_1) \slash \NS$ is the sealing forcing.

\begin{Definition}
 Let $\vec{S}=(S_i \, : \, i < \kappa)$ be a maximal antichain in
 $P(\omega_1) \slash \NS$. Then the sealing forcing for $\vec{S}$, $\seal$
 is defined as follows. Conditions
 are pairs $(p,c)$ such that $p: \alpha +1 \rightarrow \vec{S}$ and 
 $c: \alpha+1 \rightarrow \omega_1$, where the image of $c$ should be closed, and
 $\alpha < \omega_1$. We additionally demand that 
 $\forall \xi < \omega_1 \, c(\xi) \in \bigcup_{i \in \xi} p(i)$, and conditions are ordered by
 end-extension.
\end{Definition}
Thus, given a maximal antichain $\vec{S}$, $\seal$ will collapse its length down to $\omega_1$ while simultaneously shoot a club through the diagonal union of $\vec{S}$. The latter has the desired effect that $\vec{S}$ remains a maximal antichain in all stationary set preserving outer models, which is wrong if we would just collapse the length of $\vec{S}$ down to $\omega_1$. It is wellknown that $\seal$ is $\omega$-distributive and 
stationary sets preserving if and only if $\vec{S}$ is maximal.

\begin{thm}(Shelah)
 Let $V$ be a universe with a Woodin cardinal $\Lambda$. Then there is a $\Lambda$-sized 
 forcing notion $\forceP$, such that in $V^{\forceP}$, $\NS$ is saturated
 and $\omega_2 = \Lambda$.
\end{thm}
More details of the proof will be discussed later after Theorem \ref{NS saturated}

\section{Towards a proof of the theorem}
In this section we shall prove the main theorem of this work. We aim first to prove it starting from an arbitrary ground model $V$ with a Woodin cardinal. Later we will show how to reduce the parameters used in the $\Delta_1$-definition when forcing over the canonical inner model with a Woodin cardinal, $M_1$.
\begin{thm}
 Let $V$ be an arbitrary universe with a Woodin cardinal $\delta$. Let $\vec{C}$ be a ladder system on $\omega_1$ and let $\vec{T}^0$ be an $\omega$-length sequence of independent Suslin trees. Then there is  $\delta$-sized, semi-proper partial order $\forceP$ such that in the generic extension $V^{\forceP}$ of $V$ the nonstationary ideal $\NS$ is saturated and $\Delta_1$-definable over $H(\omega_2)$ with parameters $\vec{C}$ and $\vec{T}^0$.
\end{thm}
Note here that Lemma \ref{ManySuslinTrees} shows that the assumptions of the last theorem can always be met starting from a universe with a Woodin cardinal. 
The theorem has as a corollary the following result which seems to be interesting in its own right, which is why we state it explicitly.
\begin{Corollary}
The statement "$\NS$ is (boldface) $\Delta_1$-definable over $H(\omega_2)$ and saturated" is consistent with any large cardinal assumption as long as we assume the existence of a Woodin cardinal.
\end{Corollary}

We shall sketch, omitting a lot of technical issues, a simplified idea of the proof of the theorem first: We start with an arbitrary universe $V$ with a Woodin cardinal $\delta$ and an $\omega$-sequence of independent Suslin trees $\vec{T}^0$ and fix a ladder system $\vec{C}$. 
The idea is, instead of working over some generically absolute core model as is usually done in such coding arguments, to first create a suitable ground model called $W_0$ via forcing over $V$, and in a second iteration do the actual coding over $W_0$.

The universe $W_0$ will be obtained via first forcing a $\diamondsuit$-sequence $(a_{\kappa} \, : \, \kappa < \delta)$ which will witness that $\delta$ is Woodin with $\diamondsuit$. One can use $\delta$-Cohen forcing for this, which will preserve $\delta$ being Woodin and the independence of $\vec{T}^0$. Then we use a $\delta$-long, nicely supported iteration which ensures that in $W_0$ the nonstationary ideal $\NS$ is saturated, $\delta=\aleph_2$, $H(\omega_2)^{W_0}$ is a boldface $\Sigma_1$-definable object for all outer universes $W'$ of $W_0$ obtained via ccc forcings and there is a definable sequence of length $\omega_2$ of independent Suslin trees. 

In a second iteration over $W_0$ we will use the definable list of independent Suslin trees to code up all the characteristic functions of stationary subsets of $\aleph_1$ with the help of branching or specializing elements of the list. The fact that $H(\omega_2)^{W_0}$ is $\Sigma_1$-definable in all ccc extensions of $W_0$ can be utilized to define (with a $\Sigma_1$-definition) a class of suitable, $\aleph_1$-sized models of fragments of $\ZFC$ which are sufficient to define the list of Suslin trees properly and consequentially enable a boldface $\Sigma_1$ definition of stationarity. Indeed the assertion "$S$ is stationary" will be equivalent to the statement that there is a suitable model which defines some $\omega_1$ block of independent Suslin trees and this block has a pattern of branches and specializing functions which correspond to the characteristic function of $S$.

\subsection{The first iteration}
We start with $V$ as our ground model, let $\delta$ be its Woodin cardinal
and in a first step we force  the existence of a $\diamondsuit$-sequence
$(a_{\alpha} \, : \, \alpha < \delta)$ on ${V}_{\delta}$ which witnesses that $\delta$ is Woodin with $\diamondsuit$ in the following sense:

\begin{Definition}

Let $\delta$ be a Woodin cardinal then we say that $\delta$ is Woodin 
with $\diamondsuit$ iff there is a sequence $(a_{\kappa} \, :\, \kappa < \delta)$ 
such that for each $\kappa$, $a_{\kappa} \subset V_{\kappa}$ and for every 
$A \subset V_{\delta}$ the set $$\{ \kappa < \delta \,:\,
 A \cap V_{\kappa}= a_{\kappa} \land \kappa \text{ is $A$-strong up to } 
\delta \}$$ is stationary in $\delta$.
 
\end{Definition}

The $\diamondsuit$-sequence serves as our guideline for a nicely
supported iteration of length $\delta$ and can always be forced over $V$ with a $\delta$-Cohen forcing (see \cite{SchindlerNS}, Lemma 0.3).
We fix an independent $\omega$-sequence of Suslin trees $\vec{T}^0$ (w.l.o.g it exists is $V$), and an arbitrary ladder system $\vec{C}$ on $\omega_1$. We use $\vec{C}$ to define an almost disjoint family of reals $F$ of size $\aleph_1$ from it in a well-known way: working in $L[\vec{C}]$, for every $\alpha < \omega_1$ we let $r_{\alpha}$ denote the $<_{L[\vec{C}]}$-least real which codes $\alpha < \omega_1$. We consider the set $\{r_{\alpha} \cap n \, : \, n \in \omega\}$ and code every $r_{\alpha} \cap n$ as a natural number $m_{\alpha,n}$. Then, setting $s_{\alpha} = \{ m_{\alpha,n} \, : \, n \in \omega \}$ we obtain an almost disjoint family of reals of size $\aleph_1$. Note that $\{F\}$ is $\Delta_1(\vec{C})$-definable.
We emphasize that $\vec{C}$, $F$ and $\vec{T}^0$ are fixed from now on for the rest of the proof.

Starting from $V$ we begin to define an iteration of length $\delta$ of semiproper forcings which is nicely supported.
As a consequence the iteration preserves 
semiproperness (see \cite{Miyamoto}, Lemma 4.2) and the Susliness of Suslin trees in the limit steps (see \cite{Miyamoto}, Lemma 5.0).
We construct the factors by recursion:
suppose we are at stage $\alpha$ of our iteration, thus the forcing $\forceP_{\alpha}$ 
is already defined and we want to 
define the forcing $\dot{\forceQ}_{\alpha}$ from which we will
get $\forceP_{\alpha+1} = \forceP_{\alpha} \ast \dot{\forceQ}_{\alpha}$ as usual.
We define $\dot{\forceQ}_{\alpha}$ by cases
\begin{itemize}
 \item if $\alpha$ is a limit ordinal such that the $\alpha$-th entry $a_{\alpha}$ of the 
 $\diamondsuit$-sequence is the $\forceP_{\alpha}$-name of a maximal antichain $\vec{S}$
 in $P(\omega_1)\slash \NS$ of length $\ge \aleph_2$ then
 we let $\dot{\forceQ}_{\alpha}$ be the sealing forcing $\mathbb{S}(\vec{S})$, but only if 
 the forcing $\mathbb{S}(\vec{S})$ is semiproper. Otherwise force with the usual L\'evy collapse $Col(\omega_1, 2^{\aleph_2})$ (so Shelah's original proof can be adopted to the situation).
 
 \item if $\alpha \in Lim$ and $a_{\alpha}$ is the $\forceP_{\alpha}$-name of an $R \subset \omega_1$, we
 let $\dot{\forceQ}_{\alpha}$ be the almost disjoint real coding which 
 codes $R$ into a real $r_R$ relative to the 
 family $F$  of almost disjoint reals. Then we force with $Col(\omega_1, 2^{\aleph_2})$.
 \item if $\alpha \in Lim$ and $\alpha$ is such that $a_{\alpha}$ is the $\forceP_{\alpha}$-name of a real
       then force the existence of a sequence $(N_{\xi} \, : \, \xi < \omega_1)$ such that
       for a triple $(\beta, \gamma, \delta) < \omega_2$, $r$ is coded by the triple in the 
       sense of $(\ddagger)$. Then force with $Col(\omega_1, 2^{\aleph_2})$.
 \item if $\alpha$ is a limit ordinal and $a_{\alpha}$ is the $\forceP_{\alpha}$-name of a triple of limit ordinals $(\beta, \gamma, \delta) < \omega_2$
       of cofinality $\omega_1$ then force to stabilize the triple in the sense of $(\dagger)$. Then force with $Col(\omega_1, 2^{\aleph_2})$.
       
 \item if $\alpha$ is a successor ordinal then we force with Jech's forcing to obtain a Suslin tree $T_{\alpha}$. Again this is followed by forcing with $Col(\omega_1, 2^{\aleph_2})$.
 
 \item else force with the usual collapse $Col(\omega_1, 2^{\aleph_2})$.
\end{itemize}
We will list a couple of easy properties of the iteration $\forceP_{\delta} := ((\forceP_{\alpha}, \dot{\forceQ}_{\alpha}) \, : \, \alpha < \delta)$.
As nice iterations of semiproper forcings are semiproper and as every factor of the iteration is at least semiproper, this
results in a semiproper, hence stationary set preserving notion of forcing. 
Note further that the iteration has length the Woodin cardinal $\delta$, thus we will take stationarily often below $\delta$ direct limits and consequentially $\forceP_{\delta}$ has the $\delta$-c.c. Let $G$ be
the generic filter for the iteration, so that we arrive at the model $W_0:=V[G]$.
Note that every real $r \in W_0$ will have a $\forceP_{\eta}$-name for a $\eta < \delta$ and this name will be hit stationarily often by the $\diamondsuit$-sequence. Once we code a real into a triple $(\alpha, \beta, \gamma)$ in the sense of $(\ddagger)$, the triple will code $r$ in all later models of our iteration. Likewise once a triple of ordinals is stabilized, it will remain stabilized for the rest of the iteration. To summarize the things just said:

\begin{Lemma}
 $W_0= V[G]$ satisfies:
 \begin{enumerate}
\item $\aleph_2=\delta$.  
  \item Every $X \subset \omega_1$ is coded by a real with the help of our fixed almost disjoint family of reals $F$.
  \item Every triple of limit ordinals $(\alpha, \beta, \gamma) < \omega_2$ of uncountable
  cofinality is stabilized in the sense of $(\dagger)$.
  \item Every real is coded by a triple of limit ordinals $(\alpha, \beta, \gamma)< \omega_2$.
 \end{enumerate}

\end{Lemma}
  Consequentially, in $W_0$ there is a definable wellorder of $P(\omega_1)$ using the 
  fixed ladder system $\vec{C}$ and the fixed almost disjoint family $F$ as parameters. 
  
  \begin{Definition}\label{Definition Wellorder}
  Let $X, Y \in P(\omega_1)^{W_0}$ then let
  $X \unlhd Y$ if the antilexicographically least triple of ordinals $(\alpha_0, \beta_0, \gamma_0)$
  which code a real $r_0$ which codes $X$ with the help of the a.d. family $F$ is antilexicographically less or equal than 
  the antilexicographically least triple of ordinals $(\alpha_1, \beta_1, \gamma_1)$ which codes a real $r_1$ which in turn
  codes $Y$.

  \end{Definition}
Note that if $M$ is an arbitrary countable transitive model containing $\vec{C}$ such that it satisfies the items $2,3$ and $4$ from the previous Lemma, then $M$ can define a wellorder $\unlhd_M$ on its $P(\omega_1)^M$ in exactly the same way.

If we look closer we see that none of the forcings used in the iteration destroy Suslin trees,
consequentially the whole iteration preserves Suslin trees.
This is shown now in a series of Lemmas.

\begin{Lemma}\label{a.d.coding preserves Suslin trees}
 Let $T$ be a Suslin tree and let $\mathbb{A}_F(X)$ be the almost disjoint coding which codes
 a subset $X$ of $\omega_1$ into a real with the help of an almost disjoint family
 of reals of size $\aleph_1$. Then $$\Vdash_{\mathbb{A}_{F}(X)} T \text{ is Suslin }$$
 holds.
\end{Lemma}
\begin{proof}
 This is clear as $\mathbb{A}_{F}(X)$ has the Knaster property, thus the product $\mathbb{A}_{F}(X) \times T$ is ccc and $T$ must be Suslin in $V^{\mathbb{A}_{F}(X)}$. 
\end{proof}

\begin{Lemma}
 Let $\vec{S}=(S_i)_{i < \kappa}$ be a maximal antichain of stationary subsets of $\omega_1$. Let $\mathbb{S}(\vec{S})$ be
 the sealing forcing which seals off the maximal antichain. Let $T$ be a Suslin tree. Then
 $$ \Vdash_{\mathbb{S}(\vec{S})} T \text{ is Suslin}$$
 holds.
 
\end{Lemma}
\begin{proof}
 Recall first how the sealing forcing was defined. Conditions $(p,c) \in \seal$ are pairs of functions defined on successor ordinals below $\omega_1$ such that $p: \alpha+1 \rightarrow \vec{S}$ and $c: \alpha+1 \rightarrow \bigcup_{\xi < \alpha} S_{p(\xi)}$, such that $c$ is continuous and $c(\xi) \in \bigcup_{i < \xi} S_{p(i)}$ holds. We note first that without loss of generality we can demand for every condition $(p,c) \in \seal$ that $p(0)=S_0$. This has the consequence that $\seal$ is $S_0$-proper. Indeed, assume $M\prec H(\theta)$ is such that it contains $\seal$ and $(p,c) \in \seal$ and $M \cap \omega_1 \in S_0$. We shall find a $(q,d) < (p,c)$ which is $(M, \seal)$-generic. Let $(D_n)_{n \in \omega}$ be a list of the dense subsets of $\seal$ in $M$ and let $(p,c):=(p_0,c_0)>(p_1,c_1)>...$ be a descending sequence of conditions in $M \cap \seal$ which lie in the corresponding dense subset $D_i$. We can always ensure that $sup_{n \in \omega} range(c_n) = M \cap \omega_1$ which is in $S_0$. But now $(\bigcup_{n \in \omega} c_n) \cup (sup(dom(c_n)), M\cap \omega_1)$ will be a closed subset of the diagonal union of the family $\{S_{p_n (\xi_n)} \, : \, n \in \omega\}$ for $\xi_n := max (dom \, p_n)$ and it can be used in a straightforward way to produce a condition of $\seal$ which is a lower bound for the sequence $(p_n,c_n)_{n \in \omega}$, hence an $(M, \seal)$-generic condition below $p$. So $\seal$ is $S_0$-proper. 
 
Because of Lemma \ref{preservation of Suslin trees}, it is enough to show that for any regular and sufficiently large
 $\theta$, every $M \prec H_{\theta}$ with $M \cap \omega_1 = \eta \in S_0$, and every
 $p \in \mathbb{S}(\vec{S}) \cap M$ there is a $q<p$ such that for every
 $t \in T_{\eta}$, $(q,t)$ is $(M,\mathbb{S}(\vec{S}) \times T)$-generic.
 Note first that as $T$ is Suslin, every node $t \in T_{\eta}$ is an
 $(M,T)$-generic condition. Further, as forcing with a Suslin tree
 is $\omega$-distributive, $M[t]$ has the same $M[t]$-countable sets as $M$.
 By the argument of the first paragraph of the proof we know that if $M\prec H(\theta)$ is such
 that $M \cap \omega_1 \in S_0$ then an $\omega$-length descending sequence
 of $\seal$-conditions in $M$ whose domains converge to $M \cap \omega_1$
 has a lower bound as $M \cap \omega_1 \in S_0$.
 
 We construct an $\omega$-sequence of elements of $\seal$ which has a lower bound
 which will be the desired condition. 
 We list the nodes on $T_{\eta}$, $(t_i \, : \, i \in \omega)$ and
 consider the according generic extensions $M[t_i]$.
 In every $M[t_i]$ we list the $\seal$-dense subsets of $M[t_i]$,
 $(D^{t_i}_n \, : \, n \in \omega)$ and write 
 the so listed dense subsets of $M[t_i]$ as an $\omega \times \omega$-matrix and enumerate
 this matrix in an $\omega$-length sequence of dense sets $(D_i \, : \, i \in \omega)$.
 If $p=p_0 \in \seal \cap M$ is arbitrary we can find, using the fact that $\forall i \, (\seal \cap M[t_i] = M \cap \seal$), an $\omega$-length, descending
 sequence of conditions below $p_0$ in $\seal \cap M$, $(p_i \, : \, i \in \omega)$
 such that $p_{i+1} \in M \cap \seal$ is in $D_i$.
 We can also demand that the domain of the conditions $p_i$ converge to $M \cap \omega_1$.
 Then the $(p_i)$'s have a lower bound $p_{\omega} \in \seal$ and $(t, p_{\omega})$ is an
 $(M, T \times \seal)$-generic conditions for every $t \in T_{\eta}$ as any $t \in T_{\eta}$ is $(M,T)$-generic
 and every such $t$ forces that $p_{\omega}$ is $(M[T], \seal)$-generic; moreover $p_{\omega} < p$ as 
 desired.

\end{proof}

We mentioned already in Proposition \ref{MRPpreservesSuslin} that both forcings $\forceP_{\beta \gamma \delta}$ and $\forceP_{r}$ which guarantee the Caicedo-Velickovic property $(\dagger)$ and $(\ddagger)$ are forcings of the form $\forceP_{\Sigma}$ for an open stationary set $\Sigma$.  We proved already in Proposition \ref{PFAMRP} that such forcings preserve Suslin trees. 

\begin{Fact}
 Let $\forceP_{\beta \gamma \delta}$ be the forcing which stabilizes the triple $(\beta, \gamma, \delta)$ via adding a reflecting sequence $(N_i)_{i < \omega_1}$.
 Let $T$ be a Suslin tree. Then 
 $$\Vdash_{\forceP_{\beta \gamma \delta}} T \text{ is Suslin }$$ 
 does hold. The same is true for $\forceP_{r}$, the forcing which codes the real $r$ in some triple of ordinals $(\beta_r, \gamma_r, \delta_r)$ below $\omega_2$.
\end{Fact}

Since the L\'evy Collapse $Coll(2^{\aleph_2},\aleph_1)$ and the forcing $\forceP_J$ to add Suslin trees are both $\sigma$-closed, all forcing which appear in our iteration preserve Suslin trees.
\begin{Lemma}
 The iteration as defined above preserves Suslin trees: every tree $S$ which is 
 Suslin at some stage $V^{\forceP_{\alpha}}$ will remain Suslin in $W_0=V[G]$, where $G$ denotes
 the generic for the $\delta$-length iteration.
\end{Lemma}

Having established the preservation of Suslin trees we note the following: during our iteration
we cofinally often added Suslin trees $T_{\alpha}$ with the forcing $\forceP_{J}$ (where $\alpha$ denotes
the stage of the iteration where we adjoined $T_{\alpha}$). We know already that for an arbitrary finite list of trees $T_{\alpha_0}, T_{\alpha_1},...,T_{\alpha_n}$ adjoined this way, their product $T_{\alpha_0} \times T_{\alpha_1} \times...\times T_{\alpha_n}$ is a Suslin tree again, and will remain to be one in the final model $W_0$ we produce.
The definable wellorder $\unlhd$ of $P(\omega_1)$ unlocks a definition for a
canonical sequence of length $\omega_2$ of independent Suslin trees. The first entry of 
that sequence is, for technical reasons which will 
become clear later defined differently.
We start with our fixed independent $\omega$-sequence $\vec{T}^0$ and let $\vec{T}^{\alpha}$ be the $\unlhd$-least $\omega_1$ sequence of Suslin trees such that $\bigcup_{\beta < \alpha} \vec{T}^{\beta}$ concatenated with $\vec{T}^{\alpha}$ remains an independent sequence of Suslin trees.

As the wellorder $\unlhd$ in fact talks about the reals which are almost disjoint codes for the corresponding elements of $P(\omega_1)$ it will be useful to give that sequence of reals a name as well. For every element $X$ in $P(\omega_1)$, the set of reals which are almost disjoint codes for $X$ is infinite. In the following we nevertheless talk about \emph{the} real $r_X$ which codes $X \in P(\omega_1)$ by which we mean the $\unlhd$-least such real coding $X$.

\begin{Definition}
In $W_0$, 
let $(r_i \, : \, i < \omega_2)$ be the sequence of reals defined recursively
 as follows:
 
 \begin{itemize}
  \item $r_0$ is the real which codes a subset of $\omega_1$, which codes the independent $\omega$-sequence of Suslin trees $\vec{T}^0$.
 
  \item $r_{\alpha}$, for $\alpha > 0$ is the least real which is an almost disjoint code for the $\unlhd$-least subset of $\omega_1$ which itself is a code for an $\omega_1$-sequence of independent Suslin trees $\vec{T}^{\alpha}$,
  such that the concatenated sequence of the union of the Suslin trees coded in $(r_{i} \, : \, i < \alpha)$ and $\vec{T}^{\alpha}$
  forms an independent sequence again.
 \end{itemize}
 
\end{Definition}

What is very important is that this definable $\omega_2$-sequence of
independent Suslin trees will be definable in certain outer models of $W_0$.

 \begin{Lemma}\label{Definability of Suslin trees}
 Suppose that $W^{\ast}$ is a set-generic, ccc extension of $W_0$.
 Then $W^{\ast}$ is still able to define the $\omega_2$-sequence of independent
 $W_0$-Suslin trees $\vec{T}$.
\end{Lemma}
\begin{proof}
 
 Note first that if $r \in W^{\ast}$ is a real coded by a triple of ordinals in $(\alpha, \beta, \gamma)$ in
 $W^{\ast}$, then there is a reflecting sequence $(N_{\xi} \, : \, \xi < \omega_1)$ in
 $W^{\ast}$, $\bigcup_{\xi < \omega_1} N_{\xi} = \gamma$, such that 
 for club-many $\xi$, $r = \bigcup_{\eta \in (\nu, \xi)} s_{\alpha \beta \gamma} (N_{\eta}, N_{\xi})$.
 As $W^{\ast}$ is a ccc-extension of $W_0$, there is a reflecting sequence $(P_{\xi} \, : \, \xi < \omega_1)$ which is 
 an element in $W_0$, and such that $C:=\{ \xi < \omega_1 \, : \, P_{\xi} = N_{\xi}\}$ is 
 club containing in $W^{\ast}$. Indeed, every element of $(N_{\xi} \, : \, \xi < \omega_1)$ is a countable set of ordinals in $W^{\ast}$, thus can be covered by a countable set of ordinals from $W_0$.  As a consequence the sequence $(N_{\xi} \, : \, \xi < \omega_1)$ can be transformed into a continuous, increasing sequence $(P_{\xi} \, : \, \xi < \omega_1)$ in $W_0$ which coincides on its limit point with $(N_{\xi} \, : \, \xi < \omega_1)$, just as desired.
 
 But as ccc extensions preserve stationarity, the set $$\{ \zeta < \omega_1 \, : \, \exists \nu < \zeta \, ( \bigcup_{\eta \in (\nu, \zeta))}
 s_{\alpha \beta \gamma} (P_{\eta}, P_{\zeta}) = r )\}$$ which is an element of $W_0$ must contain a club from 
 $W_0$. Hence $r$ is coded by the triple $(\alpha, \beta, \gamma)$ already in $W_0$.
 
 As a consequence $P(\omega_1)^{W_0}$ is definable in $W^{\ast}$, it will be precisely
 the set of subsets of $\omega_1$ which have reals which code it with the help of 
 the almost disjoint family $F$, and such that these reals are themself
 coded by triples of ordinals below $\omega_2$ in the sense of $(\ddagger)$.
 
 Thus $W^{\ast}$ can define $W_0$-Suslin trees and our wellorder $<$ on $P(\omega_1)^{W_0}$, hence will be able to define 
 the $\omega_2$-sequence of independent Suslin trees of $W_0$.
 
\end{proof}

As a last note we emphasize that the iteration, by the theorem of S. Shelah, will ensure that in $W_0$, $\NS$ is $\aleph_2$-saturated.

\begin{thm}\label{NS saturated}
 The nonstationary ideal $\NS$ is $\aleph_2$-saturated in $W_0$.
\end{thm}
\begin{proof}[Proof Sketch]
 This is basically just a repetition of Shelah's argument. We have to justify however that the added forcings, plus the fact that we use a nicely supported iteration instead of an iteration with revised countable support do not ruin the argument. At this point it is inevitable to refer the reader to \cite{SchindlerNS}. For the following it is necessary to have an understanding of the arguments presented there. 
 
 We summarize briefly the main outline. We force to ensure that the Woodin cardinal $\delta$ in fact is Woodin with $\diamondsuit$, thus we have access to a sequence $(a_{\kappa} \, : \, \kappa < \delta)$ such that for every $A \subset V_{\delta}$:
 
 \[ \{\kappa < \delta \, : \, A \cap V_{\kappa} = a_{\kappa} \land \kappa \text{ is } A\text{-strong up to } \delta \} \]
 
 is stationary in $\delta$. 
The $\diamondsuit$-sequence serves as a bookkeeping function to define an RCS-iteration inductively as follows: assume we arrived at stage $\alpha < \delta$ and we have defined already the iteration $\forceP_{\alpha}$ up to $\alpha$.
We split into cases
\begin{enumerate}
\item If $a_{\alpha}= \tau$ and $\tau$ is the $\forceP_{\alpha}$-name of a maximal antichain in $P(\omega_1) \slash \NS$, then force with the sealing forcing $\seal (\tau)$ provided it is semiproper.
\item Otherwise force with $Col(\omega_1, 2^{\aleph_2})$.
\end{enumerate}
 The iteration has length $\delta$, and the resulting universe $V^{\forceP_{\delta}}$ satisfies that $\NS$ is saturated.
 For assume not, and let $\tau$ be a $\forceP_{\delta}$-name for an $\delta$-long antichain in $P(\omega_1) \slash \NS$.
 Then, using the fact that at inaccessible stages of the iteration we take the direct limit, there will be a stage $\kappa < \delta$ such that $a_{\kappa}= \tau \cap V_{\kappa}$ and $\tau^{G_{\kappa}}$ is a maximal antichain in $V[G_{\kappa}]$. Thus our rules tell us that the sealing forcing $\seal (\tau)^{G_{\kappa}}$ is not semiproper in $V[G_{\kappa}]$. But now an involved argument shows that the sealing forcing in fact is semiproper at that stage, which gives a contradiction.

First a couple of words to justify the usage of nicely supported iterations instead of RCS-iterations. In Shelah's proof, the following three properties of RCS-iterations are exploited:
 \begin{enumerate}
 \item RCS-iterations preserve semiproperness. 
 
 \item The tail of an RCS-iteration look like an RCS-iteration from the intermediate models point of view.

 \item At limit stages $\kappa$ which are inaccessible, if the factors of the iteration below $\kappa$ have size below $\kappa$, then at stage $\kappa$ the direct limit is taken.
 \end{enumerate}
 
 Nice iterations share the properties 1 and 2. The third item is crucially used to ensure a stage $\kappa < \delta$ such that $a_{\kappa}= \tau \cap V_{\kappa}$ and $\tau^{G_{\kappa}}$ is a maximal antichain in $P(\omega_1) \slash \NS$.
If we use a nicely supported iteration instead, we can still find such a stage $\kappa$, as the set of Mahlo cardinals is stationary below $\delta$. Indeed it suffices to intersect the unbounded set $C:= \{ \kappa < \delta \, : \, (\tau \cap V_{\kappa})^{G_{\kappa}} \text{ is a maximal antichain} \}$ with the set of Mahlos below $\delta$ to obtain the desired stage $\kappa$. Having obtained such a $\kappa$, Shelah's proof just carries over word by word to obtain the desired contradiction.

The justification for using other semiproper forcings of size $< \delta$ is even easier once one knows Shelah's argument. Again everything carries over word by word with only minimal notational alterations.

\end{proof}

This ends our discussion of the first iteration $\forceP_{\delta}$ and the crucial properties of the resulting 
model $V[G] =W_0$. In the next section we will discuss how the second iteration, starting with $W_0$ as 
the ground model does look like.

\subsection{The second iteration.}
\subsubsection{An outline of the idea}
Let us quickly describe the situation we are in. We have obtained a model $W_0=V[G]$ with the following
properties:
\begin{enumerate}
 \item In $W_0$, $\NS$ is saturated.
 \item Every subset of $\omega_1$ is coded by a real.
 \item Every real is itself coded by a triple of ordinals below $\omega_2$
 relative to the ladder system $\vec{C}$. This gives rise to a definable wellorder of $P(\omega_1)$.
 \item There is an independent $\omega_2$-length sequence $\vec{T}=(\vec{T}^{\alpha}
 \, : \, \alpha< \omega_2)$ of independent $\omega_1$-blocks $\vec{T}^{\alpha}$ of Suslin trees
 which is 
 definable over $W_0$ and which is still definable in set-generic,
 ccc-extensions of $W_0$.
\end{enumerate}
The patient reader will notice that we have not touched the issue of the definabilty of the nonstationary ideal. The second iteration is entirely concerned with coding stationary subsets of $\omega_1$. We will use the fact that a Suslin tree $T$ can be destroyed in two mutually exclusive ways using forcings with the countable chain condition: either we add a branch to $T$ or an uncountable antichain. This enables us to write the characteristic function of stationary subsets of $\omega_1$ into the definable sequence $\vec{T}$. The fact that $\vec{T}$ is an independent sequence has as a consequence that the destruction of fixed elements of $\vec{T}$ will not affect the Suslinity of the other elements of $\vec{T}$.

Thus the following strategy is promising. We start an $\omega_2$-length iteration of forcings which either specialize or shoot a branch through elements of $\vec{T}$. The iteration uses finite support.
We will enumerate in an $\omega_2$-length list all the stationary subsets of $\omega_1$ in $W_0$, pick the first stationary set listed (given by some fixed bookkeeping function) and code the characteristic function of it into the first $\omega_1$-block of our definable list of Suslin trees $\vec{T}$. 
This coding will create new stationary subsets which we list again and code up using fresh elements of $\vec{T}$ which we have not destroyed yet. Bookkeeping will yield that after $\omega_2$-many stages we will catch our tail.

The result will be a generic extension $W_0[H]$ of $W_0$ by a forcing which has the countable chain condition. 
$W_0[H]$ can define its stationary subsets of $\omega_1$ in a new way: $S \subset \omega_1$ is stationary if and only if there is an $\omega_1$-block in $\vec{T}$ (note that $\vec{T}$ is still definable in $W_0[H]$ by Lemma \ref{Definability of Suslin trees}) such that the characteristic function of $S$ is written into this $\omega_1$-block of elements of $\vec{T}$.

A calculation yields that this new definition of stationarity is boldface $\Sigma_2$ over the $H(\omega_2)$ of $W_0[H]$, thus it seems like we have not gained anything substantial. But it will be possible to slightly alter the above iteration which makes a simple definition of a class of $\aleph_1$-sized, so-called suitable models possible, which can be utilized to read off the created information on the sequence of Suslin trees. This will buy us one quantifier and we eventually arrive at a $\Sigma_1 (\vec{C}, \vec{T}^0)$-predicate for stationarity.

Of course everything just said would be rendered pointless, if $\NS$ stops to be saturated
after we perform the sketched iteration. We have to argue that $\NS$ remains saturated after forcing with the coding forcings which are ccc.

\subsubsection{Chang's Conjecture and the indestructibility of the saturation of $\NS$ under ccc forcings}

The indestructibility of the saturation of $\NS$ under ccc forcings is in fact a very delicate topic \footnote{We thank A. Lietz for pointing that out, for finding an error in an earlier version plus finding its fix.} and tightly connected to Chang's Conjecture and the semiproperness of Namba forcing as noted by S. Shelah.
First recall
\begin{Definition}
Chang's Conjecture ($\CC$) says that for every function $F:[\omega_2]^{<\omega} \rightarrow \omega_2$, there is a set $X \subset \omega_2$ of ordertype $\omega_1$ which is closed under $F$, i.e. $F"[X]^{<\omega} \subset X$.
\end{Definition}
It is known that the indestructibility of the saturation of $\NS$ follows from Chang's Conjecture (see \cite{Larson} pp. 583 for a proof).
Hence we shall argue that in $W_0$ Chang's Conjecture is true.

We take advantage of the following results which can be found in \cite{Shelah}. First recall Namba forcing, which is the canonical forcing for changing the cofinality of $\omega_2$ to $\omega$. The assertion "Namba forcing is semiproper" is known to have large cardinal strength, it is equivalent to a strong form of Chang's Conjecture \footnote{We caution the reader that the notation for various strengthenings of Chang's Conjecture is very inconsistent across the literature.}, which itself is known to be forceable from a measurable cardinal.
Semiproperness of Namba forcing can be characterized with winning strategies for player II in a certain two player game (see \cite{Shelah}, XII, Definition 2.1). 
\begin{Definition}
The two player game $\Game(\{\aleph_1 \}, \omega, \aleph_2)$ of length $\omega$ is defined as follows. At turn n, Player I  plays a function $f_n: \omega_2 \rightarrow \omega_1$ and player II responds with an ordinal $\xi_n < \omega_1$. We say that player II wins iff the set 
\[ A:= \{ i < \omega_2 \, : \, \forall n \exists m (f_n(i) < \xi_m) \} \]
is unbounded in $\omega_2$.
\end{Definition}
Shelah has shown (see \cite{Shelah} XII, Theorem 2.2, Theorem 2.5 (1)) the following results.
\begin{thm}
 Namba forcing is semiproper if and only if II has a winning strategy in $\Game(\{\aleph_1 \}, \omega, \aleph_2)$.
If player II has a winning strategy in $\Game(\{\aleph_1 \}, \omega, \aleph_2)$, then $\CC$ is true.
\end{thm}
 Before we continue one last definition of the semiproper game which is well-known to characterize semiproperness. We will use Shelah's original notation again, even though several of his parameters are meaningless in our context.
\begin{Definition}
Let $\forceP$ be a notion of forcing, $p \in \forceP$, then the semiproper Game $P\Game^{\omega}_0(p, \forceP,\aleph_1)$ is defined as follows: player I plays $\forceP$-names of countable ordinals $\dot{\zeta}_n$, $n \in \omega$, player II responds with countable ordinals $\xi_n$. In the end player II wins if and only if there is a condition $q \in \forceP$, $q<p$ and \[q \Vdash \forall n \in \omega \exists m (\dot{\zeta}_n=\xi_m)\]
\end{Definition}
As is well-known, a forcing $\forceP$ is semiproper if and only if for every $p \in \forceP$, player II has a winning strategy in $P\Game^{\omega}_0(p, \forceP,\aleph_1)$

We finally collected the relevant notions to prove that in $W_0$, Chang's Conjecture is true.  The proof of the claim owes an obvious debt to Theorem 1.9 from chapter XIII of \cite{Shelah}.
\begin{thm}
In $W_0$, Chang's Conjecture is true and consequentially the saturation of $\NS$ is indestructible by ccc forcings.
\end{thm}
\begin{proof}
The proof will make use of the following 
\begin{Claim}
Suppose that $\lambda$ is a measurable cardinal and $(\forceP_{\alpha}, \dot{\forceQ}_{\alpha} \, : \, \alpha < \lambda)$ is a nicely supported iteration such that 
\begin{enumerate}
\item $\forall \alpha < \lambda (\Vdash_{\forceP_{\alpha}} \dot{\forceQ}_{\alpha}$ is semiproper).
\item $\forall \alpha < \lambda (|\forceP_{\alpha}| < \lambda)$.
\item For every $\alpha < \lambda$, $\Vdash_{\forceP_{\alpha+1}} (2^{\aleph_2})^{V[\forceP_{\alpha}]}=\aleph_1$.
\end{enumerate}
then player II has a winning strategy in the game $\Game(\{\aleph_1 \}, \omega, \aleph_2)$ in $V[G_{\lambda}]$, where $G_{\lambda}$ is a $\forceP_{\lambda}$-generic filter over $V$. 

\end{Claim}

Once we have shown the above, we can prove the theorem as follows. The set of measurable cardinals below our Woodin $\delta$ is stationary. Moreover the iteration $\forceP_{\delta}$ we used to obtain $W_0$ satisfies the three items above. Any function $F: [\omega_2]^{<\omega} \rightarrow \omega_2$ in $V[G_{\delta}]=W_0$ has a name $\dot{F}$ of size $\delta$, and, by the $\delta$-cc of $\forceP_{\delta}$, the set of measurable $\lambda < \delta$ such that $(\dot{F} \cap V_{\lambda})^{G_{\lambda}}$ is, in $V[G_{\lambda}]$, a function from $ [\aleph_2]^{<\omega}$ to $\omega_2$ is stationary below $\delta$.
In particular, there is a $\lambda< \delta$ measurable as above and player II has a winning strategy in  $\Game (\{ \aleph_1\},\omega,\aleph_2)$ in $V[G_{\lambda}]$. So $\CC$ is true in $V[G_{\lambda}]$. Thus $\dot{F} \cap V_{\lambda}$ as evaluated by $G_{\lambda}$ has a closed subset $X$ of ordertype $\omega_1$ in $V[G_{\lambda}]$. But this $X$ also witnesses that $\CC$ is true in $W_0$.

So we shall show the claim. We fix a normal, $\lambda$-complete ultrafilter $U$ on $\lambda$.
If $\alpha < \lambda$ then, as the tail $\forceP_{\alpha,\lambda}$ is a nice iteration of semiproper forcings, we know that $\Vdash_{\forceP_{\alpha}} \forceP_{\alpha,\lambda}$ is semiproper. Hence, player II has a winning strategy (in $V^{\forceP_{\alpha}}$) in the semiproper game $P\Game^{\omega}_0(p, \forceP,\aleph_1)$ for any $p \in \forceP_{\alpha,\lambda}$.

We shall use these winning strategies  to define a winning strategy for player II in $\Game(\{\aleph_1\},\omega,\aleph_2)$ in $V[G_{\lambda}]$. Suppose the play of the game arrived at stage $n < \omega$, and player I chooses a function $f_n: \omega_2 \rightarrow \omega_1$. Then player II responds with four sets $A_n$, $\dot{f_n}$, $\alpha_n$ and $\xi_n$, where $\xi_n$ is the ordinal which is relevant as response in $\Game(\{\aleph_1\},\omega, \aleph_2)$, and which should satisfy:
\begin{enumerate}
\item $\alpha_n$ is an ordinal less than $\lambda$, but bigger than all previous $\alpha_m$, $m<n$.
\item $A_n \subset  (\alpha_n,\lambda)$, $A_n \in V[G_{\alpha_n}]$, and $A_n$ is in the ultrafilter generated by $U$ in $V[G_{\alpha_n}]$. Moreover $A_n$ is a subset of the previous $A_m$ and $A_n$ consists entirely of inaccessible cardinals $\kappa$ such that $\forall \kappa \in A_n \forall i < \kappa (|\forceP_i|< \kappa)$.
\item $\dot{f_n}$ is a $\forceP_{\lambda}$-name of $f_n$ (say the least such name in some previously fixed wellorder).
\item For every $\kappa \in A_n$, $(\dot{f}_m (\kappa), \xi_m)_{m \le n}$ is an initial segment of a play of the game $P\Game^{\omega}_0(1, \forceP_{\kappa, \lambda},\aleph_1)$ in $V[G_{\alpha_n}]$, where player II uses his winning strategy.
\end{enumerate}
We shall show that player II can always play as described above.  
Assume that for $m \le n-1$, $A_m, \dot{f_m}, \xi_m, \alpha_m$ is defined.
For every $\kappa \in A_{n-1}$, consider the instance of $P\Game^{\omega}_0 (1,\forceP_{\kappa,\lambda}, \aleph_1)$ in $V[G_{\kappa}]$:

\begin{table}[h]
\begin{tabular}{l|l}
I & $\dot{f}_0(\kappa)$ \qquad \quad $\dot{f}_1(\kappa)$ \qquad \quad $\dot{f}_2(\kappa)$ \qquad \qquad $\dots$ \qquad \qquad  $\dot{f}_n(\kappa)$ \\
\hline
II & \qquad  \qquad $\xi_0$ \qquad \qquad  $\xi_1$ \qquad  \qquad $\xi_2$ \qquad  \dots \quad $\xi_{n-1}$
   
\end{tabular}
\end{table}
Note here that we play the game for the tail forcing $\forceP_{\kappa,\lambda}$ which is still semiproper. Hence player II must have a winning strategy for each game  $P\Game^{\omega}_0 (1,\forceP_{\kappa,\lambda}, \aleph_1), \kappa \in A_{n-1}$. Let $\xi(\kappa)$ denote the $\forceP_{\kappa}$-name of an ordinal $< \omega_1$ which II has to play when following his winning strategy.

Now for every $\kappa \in A_{n-1}$, $\forceP_{\kappa}=\bigcup_{\alpha < \kappa} \forceP_{\alpha}$, as we assumed that every $A_m$ consists entirely of inaccessible cardinals and we use a nice iteration. Moreover the ultrafilter $U$ on $\lambda$ generates an ultrafilter on $\lambda$ in $V[G_{\kappa}]$, which we will denote again with $U$. As $U$ is normal, there will be a stage $\alpha_n < \lambda$ and a set $A_n \in U$, $A_{n-1} \supset A_n$, such that for every $\kappa \in A_{n-1}$, every $\forceP_{\kappa}$ name $\xi(\kappa)$ is in fact a $\forceP_{\alpha_n}$-name. We can assume that every $\xi(\kappa)$ is a nice $\forceP_{\kappa}$-name for a countable ordinal, thus there are only $<\lambda$-many such nice names, and as a consequence there is a set $A_{n} \in U$, $A_n \subset [\alpha_n,\lambda)$ and a $\forceP_{\alpha_n}$-name $\xi$ such that $\Vdash_{\alpha_n} \xi < \omega_1$, such that for every $\kappa \in A_n$, if player II plays according to her winning strategy then the play looks like

\begin{table}[h]
\begin{tabular}{l|l}
I & $\dot{f}_0(\kappa)$ \qquad \quad $\dot{f}_1(\kappa)$ \qquad \quad $\dot{f}_2(\kappa)$ \qquad \qquad $\dots$ \qquad \quad  $\dot{f}_n(\kappa)$ \\
\hline
II & \qquad  \qquad $\xi_0$ \qquad \qquad  $\xi_1$ \qquad  \qquad $\xi_2$ \qquad  \dots \quad $\xi_{n-1}$\qquad\qquad $\xi$
   
\end{tabular}
\end{table}
\begin{flushleft} We set $\xi_n:=\xi^{G_{\alpha_n}}$. It is clear that $A_n$, $\alpha_n$, $\xi_n$ and $\dot{f}_n$ are as desired.
\end{flushleft}

We shall show that the just defined strategy is indeed a winning strategy for player II in $\Game(\{\aleph_1\},\omega, \aleph_2)$ in the universe $V[G_{\lambda}]$. Let $f_0,f_1,...$ be the functions from $\omega_2$ to $\omega_1$ which are played by player I. Let $\dot{f_n}$ be the $\forceP_{\lambda}$-names. Player II plays according to the strategy defined above. In the course of the play, II will construct a sequence of quadruples $(A_n)_{n \in \omega}, (\dot{f}_n)_{n \in \omega}, (\alpha_n)_{n \in \omega}, (\xi_n)_{n \in \omega}$. We let $A_{\omega}:= \bigcap_{n \in \omega} A_n$ and note that $A_{\omega} \in U$, hence $A_{\omega}$ is unbounded in $\lambda$.
Likewise, $sup_{n \in \omega} \alpha_n < \lambda$.
If $\kappa \in A_{\omega}$, then the play

\begin{table}[h]
\begin{tabular}{l|l}
I & $\dot{f}_0(\kappa)^{G_{\lambda}}$ \qquad \quad $\dot{f}_1(\kappa)^{G_{\lambda}}$ \qquad \quad $\dot{f}_2(\kappa)^{G_{\lambda}}$ \qquad \qquad $\dots$ \\
\hline
II & \qquad  \qquad $\xi_0$ \qquad \qquad \qquad  $\xi_1$ \qquad \qquad  \qquad $\xi_2$ \qquad  $\dots$
\end{tabular}
\end{table}
is a play in $V[G_{\alpha_{\omega}}]$, where player II follows his winning strategy in $P\Game^{\omega}_0(1, \forceP_{\kappa,\lambda}, \aleph_1)$, hence for any $p \in \forceP_{\kappa \lambda}$ there is a $q< p$ such that
\[q \Vdash_{\forceP_{\kappa \lambda}} \forall n \in \omega (\dot{f}_n(\kappa) < \sup_{n \in \omega} \xi_n).\] In particular, in $V[G_{\lambda}]$, $\forall n \in \omega (f_n (\kappa) < sup_{n \in \omega} \xi_n)$.
Thus we found an unbounded set $A_{\omega} \subset \lambda=\aleph_2^{V[G_{\lambda}]}$ such that $\forall i \in A_{\omega}$ $(\forall n \in \omega \exists m \in \omega (f_n(i) < \xi_m))$, so for $\Game(\{\aleph_1\},\omega, \aleph_2)$, there is a winning strategy for player II in $V[G_{\lambda}]$, as desired.
\end{proof}

\subsubsection{Suitability}
We begin to define thoroughly how the second iteration, using $W_0$ as a ground model does look like. 
We already hinted that, in order to lower the complexity of 
a description of stationarity we need a new notion for suitable models which
will be able to define the sequence of Suslin trees $\vec{T}$ correctly.
With the notion of suitability it will become possible 
to witness stationarity already in $\aleph_1$-sized $\ZFP$ models,
as we shall see soon.
\begin{Definition}
 Let $M$ be a transitive model of $\ZFP$ of size $\aleph_1$. We say that $M$ is pre-suitable if it satisfies the 
 following list of properties:
 \begin{enumerate}
  \item $\{\vec{C}, \vec{T}^0\}  \subset M$.
  \item $\aleph_1$ is the biggest cardinal in $M$ and $M \models \forall x (|x|\le \aleph_1)$.
  \item Every set in $M$ has a real in $M$ which codes it in the sense of almost disjoint coding
  relative to the fixed family of almost disjoint reals $F$.
  \item Every real in $M$ is coded by a triple of ordinals in $M$, i.e. 
  if $r \in M$ then there is a triple $(\alpha, \beta, \gamma) \in M$ and a reflecting
  sequence $(N_{\xi} \, : \, \xi < \omega_1)\in M$ which code $r$.
  \item Every triple of ordinals in $M$ is stabilized in $M$: for $(\alpha, \beta, \gamma)$
  there is a reflecting sequence $(P_{\xi} \, : \, \xi < \omega_1) \in M$ which witnesses that
  $(\alpha, \beta, \gamma)$ is stabilized.
 \end{enumerate}

\end{Definition}

Note that the statement ''$M$ is a pre-suitable model`` is completely internal in $M$ and hence a $\Sigma_1(\vec{C}, \vec{T}^0)$-formula.
Further note that by the proof of Lemma \ref{Definability of Suslin trees}, if $W^{\ast}$ is a ccc extension of $W_0$ and $M$ is a pre-suitable
model in $W^{\ast}$ then $M \subset W_0$, as ccc extensions will not add new reflecting sequences.
\begin{Definition}
 Let $M$ be a pre-suitable model. We say that $M$ is $W_0$-absolute for Susliness
 if $T \in M$ is an element from $W_0$ and $M \models T \text{ is Suslin}$, then $T$ is Suslin in $W_0$.
 Likewise we say that $M$ is $W_0$-absolute for stationarity if
 $S \in M$ and $M$ thinks that
 $S$ is a stationary subset of $\omega_1$ then $S$ is a stationary subset of $\omega_1$ in $W_0$. 
 A pre-suitable model which is $W_0$-absolute for stationarity and Susliness is called suitable.
\end{Definition}

We have already seen in Lemma \ref{Definability of Suslin trees} that ccc extensions of $W_0$ will still be able to define our
$\omega_2$-sequence of independent $W_0$-Suslin trees $\vec{T}$. With the notion of suitability
we can localize this property in the following sense:
\begin{Lemma}
 Let $W^{\ast}$ be a ccc extension of $W_1$, and let $M \in W^{\ast}$ be a suitable model.
 If $M$ computes the $\omega_2$-length sequence of independent Suslin trees from $W_0$ using its local wellorder $\unlhd_M$,
 then the computation will be correct, i.e. $\vec{T}^M = \vec{T} \cap M$.
\end{Lemma}
\begin{proof}
 We shall show inductively that for every $\eta \in M$, the $\eta$-th
 block of $\vec{T}$, $\vec{T}^{\eta}$ will be computed correctly by $M$ (if an element of $M$).
 For $\eta=0$ this is true as $\vec{T}^0$ is by definition of suitability an element of $M$.
 Let $\eta \in M$ and assume by induction that the sequence $\vec{T}$ up to the $\eta$-th block is computed correctly by $M$. We shall show that $M$ computes the $\eta+1$-th block correctly. Recall that $\vec{T}^{\eta+1}$ was defined to be the $\unlhd$-least $\omega_1$-block of independent Suslin trees such that $\bigcup_{\beta \le\eta} \vec{T}^{\beta}$ concatenated with $\vec{T}^{\eta+1}$ remains an independent sequence in $W_0$. 
 
 Assume for a contradiction that the suitable $M$ computes $\vec{T}'$ as its own different version of $\vec{T}^{\eta+1}$. Thus there is a real
 $r'$ which codes $\vec{T}'$, and $r'$ itself is coded into a triple of $M$-ordinals $(\alpha', \beta', \gamma')< \omega_2^{M}$. Let $r_{\eta+1}$ be the real which codes $\vec{T}^{\eta+1}$.
 We claim that $r_{\eta+1} < r'$ by which we mean that the least triple of ordinals which codes $r_{\eta+1}$ is antilexicographically less than the least triple which codes $r'$. Otherwise $r'< r_{\eta+1}$ and by the Suslin-absoluteness of $M$
 the independent-$M$-Suslin trees coded into $r'$ would be an independent $\omega_1$-sequence of Suslin trees in $W_0$,
 moreover they would still form an independent sequence when concatenated with $\vec{T}^{\eta}$ in $W_0$ which
 is a contradiction to the way $\vec{T}^{\eta+1}$ was defined.
 
 So $r_{\eta+1}<r'$, so the least triple of ordinals $(\alpha, \beta, \gamma)$ coding $r_{\eta+1}$
 is antilexicographically less than $(\alpha', \beta', \gamma')$. Note that the suitability of $M$ implies
 that $(\alpha, \beta, \gamma)$ is stabilized in $M$. Thus there is a reflecting sequence
 $(P_{\xi} \, : \, \xi < \omega_1)$ in $M$ witnessing this. As $W^{\ast}$ is a ccc extension of $W_0$ we can assume
 that the sequence is in fact an element of $W_0$. At the same time there is a reflecting 
 sequence $(N_{\xi} \, : \, \xi < \omega_1)$ in $W_0$ which witnesses that $r_{\eta+1}$ is coded by $(\alpha, \beta, \gamma)$.
 By the continuity of both sequences, there is a club $C$ in $W_0$ such that
 $\forall \xi \in C ( N_{\xi} = P_{\xi}).$ Thus the limit points of $C$ witness that
 in fact the sequence $(P_{\xi} \, : \, \xi < \omega_1)\in M$ codes $r_{\eta+1}$ as well but the club
 $C$ is in $W_0$, so we need an additional argument to finish. Recall that the suitability of $M$
 implies that $M$ is absolute for stationarity, thus if the set 
 $\{ \xi < \omega_1 \, : \, \exists \nu < \xi ( \bigcup_{\zeta \in (\nu, \xi)} s_{\alpha \beta \gamma} (P_{\zeta}, P_{\xi}) \ne r_{\eta+1}) \}$
 would be stationary in $M$ it would be stationary in $W_0$ which is a contradiction.
 So $M$ computes $r_{\eta+1}$ correctly and the rest of the inductive argument can be repeated exactly as 
 above to show that $\vec{T}^M = \vec{T} \upharpoonright (M \cap Ord)$ as desired.
 
  \end{proof}

So suitable models will compute $\vec{T}$ correctly, and if we start to write information into $\vec{T}$, a suitable model can be used to read it off. In $W_0$ cofinally many suitable model below $\omega_2$ exist.

\begin{Lemma}
Work in $W_0$. If we set
$$ P:= \{ \eta < \omega_2 \, : \, \exists M (M \text{ is suitable } \land M \cap Ord = \eta\}$$ then $P$ is unbounded in $\omega_2$.
\end{Lemma}
\begin{proof}
Recall the iteration $(\forceP_{\alpha} \, : \, \alpha < \delta)$ to force $W_0$. Whenever we are at an intermediate stage $\kappa < \delta$ such that $\kappa$ is Mahlo then, if $G_{\kappa}$ denotes the generic filter for $\forceP_{\kappa}$, $H(\omega_2)^{V[G_{\kappa}]}$ will be a pre-suitable model.
Moreover, by the properties of nicely supported iterations, stationary subsets of $\omega_1$ in $H(\omega_2)^{V[G_{\kappa}]}$ will remain stationary in $W_0$ and Suslin trees in $H(\omega_2)^{V[G_{\kappa}]}$ will remain Suslin trees in $W_0$, as the tail iteration $\forceP_{[\kappa, \delta)}$ is a nicely supported iteration of semiproper, Suslin tree preserving notions of forcing over the ground model $V[G_{\kappa}]$. Thus every  $H(\omega_2)^{V[G_{\kappa}]}$, for $\kappa$ Mahlo, is a suitable model which gives the assertion of the Lemma.

\end{proof}

We shall use forcing to obtain a universe in which the just defined set of suitable models 
\begin{itemize}
\item[$U'$]
$:=\{ M \, : \, \exists \kappa < \delta (\kappa \text{ is Mahlo } \land M=H(\omega_2)^{V[G_{\kappa}]} \}$
\end{itemize} 
becomes easily definable. 
This will be our first forcing in the second iteration. We note that every $M \in U'$ is itself coded by a real relative to the almost disjoint family $F$. We want to use  the variant of almost disjoint coding forcing to code up  \begin{itemize}
\item[$U$] $:= \{ r_M \, : \, r_M $ is the least almost disjoint code for a subset of $\omega_1$ which codes $M \in U' \}$
\end{itemize}
into one real $r_{U}$. Recall that $\mathbb{A}(U)$ is a forcing of size $|U|=\delta$ which is Knaster and which adds a real $r_U$ such that in $W_0^{\mathbb{A}(U)}$ the following holds:
\begin{itemize}
\item[$(\ast)$] $\forall x \in 2^{\omega} \cap W_0 \,( x \in U \leftrightarrow r_U \cap S(x)$ is finite$)$.
\end{itemize}
In a next step we will code the characteristic function of the real $r_U$ into a pattern on $\vec{T}^0$. We use the fact that a Suslin tree can generically be destroyed in two mutually exclusive ways. We can generically add a branch or generically add an antichain without adding a branch.
We fix the first $\omega$-block of independent Suslin trees $\vec{T}^0$ and we let 
$\mathbb{C} := \prod_{n \in \omega} \forceP_n$ with finite support, where
\begin{equation*}
\forceP_n = \begin{cases}
T^{0}_n  & \text{ if } r_U(n)=1 \\ Sp(T^{0}_n) & \text{ if } r_U(n)=0
\end{cases}
\end{equation*}
and $T^{0}_n$ denotes the forcing notion one obtains when forcing with nodes of $T^{0}_n$ as conditions, and
$Sp(T^{0}_n)$ denotes Baumgartner's forcing which specializes $T^{0}_n$ with finite conditions and which is known to be ccc (see \cite{Baumgartner}).
Iterations of the just described form always have the countable chain condition.

\begin{Lemma}\label{ccc Coding}
Let $\vec{T}= (T_{\alpha} \, : \, \alpha < \eta)$ be an independent sequence of Suslin trees of length $\eta$. Let $f: \eta \rightarrow 2$ be an arbitrary function and let $T_{\alpha}$ also denote the partial order when forcing with the tree $T_{\alpha}$ and let $Sp(T_{\alpha})$ be the forcing which specializes the tree $T_{\alpha}$.

Then if we consider the finitely supported product $\mathbb{C} := \prod_{\beta < \eta} \forceP_{\beta}$ where
\begin{equation*}
\forceP_{
\beta } = \begin{cases}
T_{\beta}  & \text{ if } f(\beta)=1 \\ Sp(T_{\beta}) & \text{ if } f(\beta)=0
\end{cases}
\end{equation*}
then $\mathbb{C}$ has the countable chain condition.
\end{Lemma}
\begin{proof}
Fix an arbitrary $f: \eta \rightarrow 2$. We prove the Lemma using induction over the length $\eta$. The limit case is true as we use finite support. Thus assume the assertion of the Lemma is true for $\eta$ and we want to show it is true for products of length $\eta+1$. Assume for a contradiction that $\prod_{\alpha < \eta+1}\forceP_{\alpha}$ (according to $f$) does not have the countable chain condition.
Hence the tree $T_{\eta}$ is not a Suslin tree in $V^{\prod_{\alpha < \eta} \forceP_{\alpha}}$, as otherwise both forcings $T_{\eta}$ and $Sp(T_{\eta})$ would have the countable chain condition.
But $\prod_{\alpha < \eta} \forceP_{\alpha} \ast T_{\eta} = \prod_{\alpha < \eta} \forceP_{\alpha} \times T_{\eta} = T_{\eta} \times \prod_{\alpha < \eta} \forceP_{\alpha}$, and the latter is a forcing with the countable chain condition. Indeed as $T_{\eta}$ does not touch the Susliness of any member of $\vec{T}$ besides $T_{\eta}$,  $(T_{\alpha} \, : \, \alpha < \eta)$ is an independent sequence of Suslin trees in $V^{T_{\eta}}$, thus
by induction hypothesis,  $\prod_{\alpha < \eta} \forceP_{\alpha}$ has the countable chain condition in $V^{T_{\eta}}$, so it has the countable chain condition in $V$.
\end{proof}

In particular this means that $\mathbb{C}$ as defined above is a forcing with the countable chain condition which writes the real $r_U$ into a pattern of 0 and 1's on the sequence $\vec{T}^0$ of Suslin trees. We
let $H_0$ be
a generic filter for $\mathbb{A}_U$ over $W_0$ and
let
 $H_1$ denote a $W_0[H_0]$-generic filter for $\mathbb{C}$. The resulting model $W_1:=W_0[H_0][H_1]$ is the ground model for a second iteration we define later.
 
\begin{Lemma}
 Let $W_1$ be the universe $W_0[H_0][H_1]$

 Then
 $W_1:= W_0[H_0][H_1]$ is a ccc extension of $W_0$ which satisfies:
 \begin{itemize}
 \item $ r_U(n)=1$ if and only if $T^{0}_n$ has a branch.
 \item $ r_U(n)=0$ if and only if $T^{0}_n$ is special.
\end{itemize}
\end{Lemma}

\begin{proof}
It suffices to show that $W_0[H_0]$ is a Suslin-tree-preserving extension of $W_0$. This is clear as $\mathbb{A}(U)$ is Knaster. 

\end{proof}
To summarize we arrived at a situation where the real $r_U$, which captures all the information about a set of suitable models, is written into the sequence of Suslin trees $\vec{T}^0$. Consequentially, any transitive, $\aleph_1$-sized model  $M$ which contains $\vec{C}$, $\vec{T}^0$ and which sees that every tree in $\vec{T}$ is destroyed can compute the real $r_U$ and thus has access to a set of suitable models, which in turn can be used to compute $\vec{T}$ inside $M$ in a correct way. This line of reasoning remains sound in all outer ccc extensions of $W_1$ as we shall see. Thus the ability of finding $\vec{T}$ in suitable models gives rise to the possibility of using $\vec{T}$ for additional coding arguments over $W_1$.

\begin{Lemma}\label{sigma_1 set of suitable models}
 Let $W^{\ast}$ be a ccc extension of $W_1$ then there is a $\Sigma_1(\vec{C}, \vec{T}^0)$-formula $\Phi(v)$ such that whenever $x \in 2^{\omega}$ and $W^{\ast} \models \Phi(x)$ then $x$ is the almost disjoint code for a suitable model. 
\end{Lemma}
\begin{proof}
The formula $\Phi(x)$, is defined as follows:

 \begin{itemize}
 \item[$\Phi(x)$]   if and only if there is a transitive, $\aleph_1$-sized $\ZFP$ model $N$ which contains $\vec{C}$ and $\vec{T}^0$ such that the following holds in $N$:  
   \begin{itemize}
    \item $N$ sees a full pattern on $\vec{T}^0$, i.e. for every $n \in \omega$ and 
   every member $T^0_{n}$, $N$ has either a branch through $T^0_{n}$ 
   or a function which specializes $T^0_{n}$. 
\item  The pattern on $\vec{T}^0$ corresponds to the characteristic function of a real $r$ and $N$ thinks that there is a pre-suitable $N'$ such that $x \in N'$ and $(S(x) \cap r)$ is a finite set.
 \end{itemize}
   \end{itemize}
This is a $\Sigma_1(\vec{C}, \vec{T}^0)$-formula, as it is of the form $\exists N( N\models ...)$. 
We shall show that whenever $x$ is a real from $W^{\ast}$ such that $\Phi(x)$ holds then $x$ is the almost disjoint code for a suitable model.
Recall the set $U'$ and $U$ we defined above.
Note first that $N$ has access to the set $U$ which is the set of reals which are codes for the set of suitable models $U'$.
This is clear as $\vec{T}^0 \in N$, thus if $N$ sees a pattern on $\vec{T}^0$, this pattern must be the unique pattern from $W_1$, which corresponds to the characteristic function of the real $r_U$. The statement "$N'$ is a pre-suitable model" is a $\Delta_1(\vec{C})$ formula in $N'$, thus absolute for the transitive $N$. So if $N$ thinks that there is a pre-suitable $N'$ which contains $x$ as an element, then this is true in $W^{\ast}$. As $W^{\ast}$ is a ccc extension of $W_1$ and hence of $W_0$ as well, we know already that we can express the statement "$y \in 2^{\omega} \cap W_0$" in $W^{\ast}$ as "$\exists P (P$ is pre-suitable and $y \in P)$". Thus if $N$ thinks that there is a pre-suitable $N'$ which contains $x$ as an element, then $x \in W_0$, and now $(\ast)$ from above applies to conclude that indeed $x \in U$ which is what we wanted.
\end{proof}
\subsubsection{Coding stationarity}
We work now over $W_1$ and start a finite support iteration $(\forceQ_{\alpha} \, : \, \alpha < \omega_2)$ of length $\omega_2=\delta$. We pick the G\"odel pairing function as our bookkeeping function $F: \omega_2 \times \omega_2 \rightarrow \omega_2$. Recall that $F$ is a bijection with the additional property that $\forall \beta < \omega_2( \beta \ge max((\beta_1, \beta_2))= F^{-1} (\beta)$. The iteration is defined recursively as follows. Assume we are at stage $\alpha < \omega_2$ and we have already defined $\forceQ_{\alpha}$. Let $H_{\alpha}$ be the generic filter for $\forceQ_{\alpha}$ over $W_1$. We also assume inductively that the $\alpha$-th $\omega_1$-block $\vec{T}^{\alpha}$ is still an independent sequence of Suslin trees in $W_1[H_{\alpha}]$, and that $W_1[H_{\alpha}] \models |P(\omega_1) \slash \NS)| = \aleph_2$. We want to describe the next forcing notion $\dot{\forceQ}_{\alpha}$ we will use.

Let $(\alpha_1, \alpha_2) = F^{-1}(\alpha)$. We consider the universe $W_1[H_{\alpha_1}]$ and a fixed enumeration $(S^{\alpha_1}_i \, : \, i < \omega_2)$ of the stationary subsets of $\omega_1$ in $W_1[H_{\alpha}]$.  Then we pick the  $\alpha_2$-th element $S^{\alpha_1}_{\alpha_2}$ of $(S^{\alpha_1}_i \, : \, i < \omega_2)$ and code the characteristic function of  $S^{\alpha_1}_{\alpha_2}$ into the $\alpha$-th $\omega_1$ block $\vec{T}^{\alpha}= (T^{\alpha}_i \, : \, i < \omega_1)$ of our fixed independent sequence of Suslin trees $\vec{T}$. 

To be more precise we let $r^{\alpha_1}_{\alpha_2}: \omega_1 \rightarrow 2$ denote the characteristic function of $S^{\alpha_1}_{\alpha_2}$. At stage $\alpha$ we will force with
$\dot{\forceQ}_{\alpha} := \prod_{\beta < \omega_1} \forceP_{\beta}$ with finite support, where
\begin{equation*}
\forceP_{
\beta } = \begin{cases}
T^{\alpha}_{\beta}  & \text{ if } r^{\alpha_1}_{\alpha_2}(\beta)=1 \\ Sp(T^{\alpha}_{\beta}) & \text{ if } r^{\alpha_1}_{\alpha_2}(\beta)=0
\end{cases}
\end{equation*}
We let $h_{\alpha+1}$ be a generic filter for $\dot{\forceR}_{\alpha}$ over $W_1[H_{\alpha}]$, set $H_{\alpha+1}= H_{\alpha} \ast h_{\alpha+1}$ and continue. Note that $\dot{\forceQ}_{\alpha}$ has the countable chain condition as we assumed that $\vec{T}^{\alpha}$ is an independent sequence of Suslin trees in $W_1[H_{\alpha}]$ and by Lemma \ref{ccc Coding}. We shall prove now the property of $(\forceQ_{\alpha} \, : \, \alpha < \omega_2)$ which was assumed inductively in the definition of the iteration, namely that tails of $\vec{T}$ always remain an independent sequence of Suslin trees.

\begin{Lemma}
Let $(\forceQ_{\alpha} \, : \, \alpha \le \omega_2)$ be the just defined iteration over $W_1$ and let $\vec{T}$ be our sequence of independent Suslin trees. We write $\vec{T}^{>\alpha}$ for $(\vec{T}^{\beta} \,  :  \, \alpha < \beta < \omega_2)$. For every $\alpha < \omega_2$, if $H_{\alpha}$ is a generic filter for $\forceQ_{\alpha}$ then $W_1[H_{\alpha}]$ is a ccc extension of $W_1$ and $$ W_1[H_{\alpha}] \models \vec{T}^{>\alpha} \text{ is an independent sequence of (blocks of) Suslin trees.}$$
Moreover at stage $\alpha$ every element of $\vec{T}^{\gamma}$, $\gamma < \alpha$ has been destroyed.
\end{Lemma}
\begin{proof}
We prove the Lemma via induction on $\alpha < \omega_2$. For $\alpha=0$ the Lemma is true. Assume now that $\forceQ_{\alpha}$ is a ccc forcing and $\vec{T}^{>\alpha}$ is an independent sequence of Suslin trees in $W_1^{\forceQ_{\alpha}}$. We shall show that $\forceQ_{\alpha+1}=\forceQ_{\alpha} \ast \dot{\forceQ}_{\alpha}$ (where $\dot{\forceQ}_{\alpha}$ is an $\omega_1$-length product which codes some set in $P(\omega_1)^{W_1[G_{\alpha}]}$) is ccc and that $T^{>\alpha+\omega_1}$ is an independent sequence of Suslin trees in $W_1^{\forceQ_{\alpha+1}}$.

That $\forceQ_{\alpha+1}=\forceQ_{\alpha} \ast \dot{\forceQ}_{\alpha}$ has the ccc follows from our inductive hypothesis and Lemma \ref{ccc Coding}, applied in the universe $W_1^{\forceQ_{\alpha}}$.

That $\vec{T}^{> \alpha + \omega_1}$ remains independent in $W_1^{\forceQ_{\alpha+1}}$  follows in a similar way. Assume not,
then there is an $n \in \omega$ and $\{T_i \, : \, i \in n \} \subset \vec{T}^{>\alpha+\omega_1}$, and $T:=T_0 \times ... \times T_{n-1}$ is not a Suslin tree any more in $W_1^{\forceQ_{\alpha+1}}$, hence $T$ does not have the ccc in $W_1^{\forceQ_{\alpha+1}}$. But then $\forceQ_{\alpha} \ast (\dot{\forceQ}_{\alpha} \times T)$ is not ccc which contradicts our inductive hypothesis and Lemma \ref{ccc Coding}.

Assume now that $\alpha$ is a limit ordinal and that the assertion is true for every $\beta < \alpha$. First note that $\forceQ_{\alpha}$ must have the ccc by our inductive hypothesis  and as we use finite support.

To show that $\vec{T}^{>\alpha}$ remains independent in $W_1^{\forceQ_{\alpha}}$, assume the opposite. Thus there is an $n \in \omega$ and $\{T_i \, : \, i \in n \} \subset \vec{T}^{>\alpha}$, and $T:=T_0 \times ... \times T_{n-1}$ is not a Suslin tree in $W_1^{\forceQ_{\alpha}}$. In particular $\forceQ_{\alpha} \times T$ does not have the ccc.  By our inductive hypothesis for very $\beta < \alpha$, $T$ is Suslin in $W_1^{\forceQ_{\beta}}$, hence $\forceQ_{\beta} \times T$ is ccc, thus the direct limit of $(T \times \forceQ_{\beta})_{\beta < \alpha}$, which is isomorphic to $T \times \forceQ_{\alpha}$ must have the ccc as well. This is a contradiction as $T \times \forceQ_{\alpha} = \forceQ_{\alpha} \times T$.

\end{proof}
As a consequence the iteration $(\forceQ_{\alpha} \, : \, \alpha \le \omega_2)$ is a finite support iteration of forcings with the countable chain condition, thus $\forceQ_{\omega_2}$ has the countable chain condition as well.
This, plus the fact that $W_1\models 2^{\aleph_0} = 2^{\aleph_1} = \aleph_2$  readily gives that for every $\alpha < \omega_2$, if $H_{\alpha}$ denotes the generic filter for $\forceQ_{\alpha}$, then $W_1[H_{\alpha}] \models 2^{\aleph_1} = \aleph_2$. Indeed if $X \in W_1[H_{\alpha}] \cap P(\omega_1)$, then $X$ has a $\forceQ_{\alpha}$-name which is uniquely determined by a sequence $(A_{\eta} \, : \, \eta < \omega_1)$ such that every  $A_{\eta} \subset \forceQ_{\alpha}$ is a maximal antichain. There are $[\aleph_2]^{\aleph_0} = \aleph_2$-many antichains in $\forceQ_{\alpha}$, thus there are $\aleph_2$-many $\forceQ_{\alpha}$-names for subsets of $\omega_1$. As a consequence we will catch our tail after $\omega_2$ many stages.

\begin{Lemma}
Let $H_{\omega_2}$ be an $W_1$-generic filter for $(\forceQ_{\alpha} \, : \, \alpha < \omega_2)$. Then in $W_1[H_{\omega_2}]$, every stationary subset of $\omega_1$ is coded into an $\omega_1$-block of $\vec{T}$, i.e. for each $S$ stationary there is an $\alpha < \omega_2$ such that for every $\beta < \omega_1$ ( $\vec{T}^{\alpha}_{\beta}$ has a branch if and only if $\beta \in S$ and $\vec{T}^{\alpha}_{\beta}$ is special if and only if $\beta \notin S$). Here we write  $\vec{T}^{\alpha}_{\beta}$ for the $\beta$-th element of the $\alpha$-th $\omega_1$-block of Suslin trees $\vec{T}^{\alpha}$ from $\vec{T}$.
\end{Lemma}
Recall that as $W_1[H_{\omega_2}]$ is a ccc extension of $W_0$, $H(\omega_2)^{W_0}$ is definable in $W_1[H_{\omega_2}]$ via the formula \begin{itemize}
\item[] $x \in H(\omega_2)^{W_0}$ if and only if $\exists M (M$ is presuitable and $x \in M)$.
\end{itemize}
As $\vec{T}$ is definable over $H(\omega_2)^{W_0}$, $\vec{T}$ is definable in $W_1[H_{\omega_2}]$ as well and so it can see the pattern that was written into $\vec{T}$. This hands us a new, definable predicate for stationary subsets of $\omega_1$. Counting quantifiers yields that the statement of the last Lemma is $\Sigma_2(\vec{C})$ over $H(\omega_2)$. Using the $\Sigma_1(\vec{C}, \vec{T}^0)$-definable set of suitable models from Lemma \ref{sigma_1 set of suitable models} will yield a $\Sigma_1(\vec{C}, \vec{T}^0)$-definition for stationarity in $W_1[H_{\omega_2}]$.
\begin{Lemma}
There is a $\Sigma_1(\vec{C}, \vec{T}^0)$-formula $\Psi(S)$ which defines stationary subsets of $W_1[H_{\omega_2}]$:
\begin{itemize}
\item[$\Psi(S)$] if and only if there is an $\aleph_1$-sized, transitive model $N$ which contains $\vec{C}$ and $\vec{T}^0$ such that $N$ models that
\begin{itemize}
\item There exists a real $x$ such that $\Phi(x)$ holds, i.e. $x$ is a code for a suitable model $M$.
\item There exists an ordinal $\alpha$ in the suitable model $M$ such that $\vec{T}'$ is the $\alpha$-th $\omega_1$ block of the definable sequence of independent Suslin trees as computed in $M$ and $N$ sees a full pattern on $\vec{T}'$.
\item $\forall \beta < \omega_1$ $( \beta \in S$ if and only if $\vec{T}' (\beta)$ has a branch).
\item $\forall \beta < \omega_1$ $(\beta \notin S$ if and only if $\vec{T}' (\beta)$ is special).
\end{itemize}
\end{itemize}
Note that $\Psi(S)$ is of the form $\exists N (N \models...)$, thus $\Psi$ is a $\Sigma_1$-formula.
\end{Lemma}

\begin{proof}
We assume first that $S \subset \omega_1$ is a stationary set in $W_1[H_{\omega_2}]$. Then by the last Lemma, $S$ is coded into an $\omega_1$-block of $\vec{T}$. The set of branches and specializing functions which witness that $S$ is coded into $\vec{T}$ is a set of size $\aleph_1$, thus there will be a suitable model $M$ which will contain it. We ensured that this suitable $M$ is coded by a real $x$ in $W_1[H_{\omega_2}]$. Then any uncountable, transitive $N$ which contains $r$ is a witness for $\Psi(S)$.

If, on the other hand $\Psi(S)$ is true then the suitable model $M$ which contains $S$ and witnesses that the characteristic function of $S$ is written into an $\omega_1$-block of its local $M$-Suslin trees is also a witness that the characteristic function of $S$ can be found in some block of $\vec{T}$ in the real world $W_1[H_{\omega_2}]$. This suffices as by definition of the iteration$ (\forceQ_{\alpha}\, : \, \alpha<  \omega_2)$, the patterns on $\vec{T}$ code up stationarity on $\omega_1$.
\end{proof}
This finishes the proof of the main theorem. We add as a remark that the coding procedure is actually more general and is not restricted to code $\NS$. Indeed, fixing $W_0$ as our ground model, we can pick any subfamily $\mathcal{F}$ of $P(\omega_1)$ and can code its members into the independent sequence $\vec{T}$. If $\mathcal{F}$ is definable and is preserved under ccc extensions (such as $\mathcal{F}= \NS^{+}$), then our proof works with almost no alterations. 

\section{Forcing over $M_1$}
As stated already in the beginning of this paper, one can apply the just introduced coding method to the canonical inner model with one Woodin cardinal $M_1$, and obtain better results in terms of the parameters which are used in the definition of stationarity. We will introduce some of the properties of $M_1$ which are crucial for our needs, but assume from this point on that the reader is familar with the basic notions of inner model theory. Recall that $M_1$ is a proper class premouse containing a Woodin cardinal (see \cite{Steel2}, pp. 81 for a definition of $M_1)$. Every initial segment $\mathcal{J}^{M_1}_{\beta}$ is $\omega$-sound and 1-small, where we say that a premouse $\mathcal{M}$ is 1-small iff whenever $\kappa$ is the critical point of an extender on the $\mathcal{M}$-sequence then
\[ \mathcal{J}^{\mathcal{M}}_{\kappa} \models \lnot \exists \delta (\delta \text{ is Woodin}).\]
The reals of $M_1$ admit a $\Sigma^{1}_3$-definable wellorder (see \cite{Steel2}, Theorem 4.5), the definition of the wellorder makes crucial use of a weakened notion of iterability, the so-called $\Pi^1_2$-iterability which we shall introduce.

Let $\mathcal{M}$ be a premouse, $\mathcal{T}$ be an $\omega$-maximal iteration tree $b$ a branch through $\mathcal{T}$ and $\alpha$ an ordinal. Then $b$ is $\alpha$-good if, whenever $\mathcal{N}=\mathcal{M}^{\mathcal{T}}_b$ or $\mathcal{N}$ is the $\alpha$-th iterate of some initial segment $\mathcal{P} \trianglelefteq \mathcal{M}^{\mathcal{T}}_b$ using a single extender $E$ (and its images under the iteration map) on the $\mathcal{P}$-sequence, then $\alpha$ is in the wellfounded part of $\mathcal{N}.$ Then we say that a premouse $\mathcal{M}$ is $\Pi^1_2$-iterable, if player II has a winning strategy in the game $\mathcal{G}'_{\omega}(\mathcal{M},1)$, where $\mathcal{G}'_{\omega}(\mathcal{M},1)$, is defined just as the ordinary weak two player game $W\mathcal{G}_{\omega}(\mathcal{M},1)$ (see e.g. \cite{Steel3} pp. 65 for a definition), with the exception that player I not only plays an $\omega$-maximal, countable putative iteration tree $\mathcal{T}$ but additionally has to play a countable ordinal $\alpha< \aleph_1^{M_1}$. Then player II does not have to play a wellfounded branch through $\mathcal{T}$ (as it would be the case for iterability), but instead can play a cofinal branch $b$ through $\mathcal{T}$ such that $b$ is $\alpha$-good in order to win.

The winning strategy for II for $\mathcal{G}'_{\omega}(\mouseM,1)$ guarantees that
$\mouseM$ can be compared to any countable premouse which is an initial segment of $M_1$, as we shall prove below. The proof works for any $\omega_1$-preserving generic extension of $M_1$, which is what we need for our purposes.

We briefly introduce a couple of fundamental inner model theoretic notions which will play a crucial role in the arguments to come.

\begin{Definition}
Let $\mathcal{T}$ be a $k$-maximal iteration tree of limit length on some premouse $\mathcal{M}$. Then we let
\[ \delta(\mathcal{T}) = sup \{E^{\mathcal{T}}_{\alpha} \, : \, \alpha < lh(\mathcal{T}  \} \]
Moreover we let 
\begin{align*}
\mathcal{M}(\mathcal{T}):= \text{ unique passsive } \mathcal{P} \text{ such that } Ord^{\mathcal{P}}=\delta(\mathcal{T}) \text{ and} \\
\forall \alpha < \delta(\mathcal{T}) (\mathcal{M}(\mathcal{T}) \text{ agrees with } \mathcal{M}^{\mathcal{T}}_{\alpha} \text{ below } lh(E^{\mathcal{T}}_{\alpha})).
\end{align*}

\end{Definition}

We also need the so-called zipper lemma (see \cite{Steel3}, Theorem 6.10).

\begin{thm}
Let $c$ and $c'$ be distinct cofinal branches of the $k(\le \omega$)-maximal iteration tree $\mathcal{T}$ on the premouse $\mathcal{M}$. Let $A \subset \delta(\mathcal{T})$ and assume that $\delta(\mathcal{T}),A \in $wfp$(\mathcal{M}^{\mathcal{T}}_c) \cap $wfp$(\mathcal{M}^{\mathcal{T}}_{c'})$. 
Then \[ \mathcal{M}^{\mathcal{T}}_c \models \exists \kappa < \delta(\mathcal{T})(\kappa \text{ is } A\text{-strong up to } \delta(\mathcal{T}))\]
\end{thm}
Last we recall what $\mathcal{Q}$-structures are.
\begin{Definition}
Let $\mathcal{M}$ be a premouse, $\mathcal{T}$ be a $k (\le \omega)$-maximal iteration tree on $\mathcal{M}$, let $b$ be a cofinal, wellfounded branch through $\mathcal{T}$ and let $\gamma$ be the least ordinal, if there is one, such that either
\[ \omega \gamma < On^{\mathcal{M}^{\mathcal{T}}_{b}} \text{ and } \mathcal{J}^{\mathcal{M}^{\mathcal{T}}_{b}}_{\gamma+1} \models \delta (\mathcal{T}) \text{ is not Woodin, } \] 
or
\[ \omega \gamma = On^{\mathcal{M}^{\mathcal{T}}_{b}} \text{ and } \rho_{n+1} (\mathcal{J}^{\mathcal{M}^{\mathcal{T}}_{b}}_{\gamma}) < \delta(\mathcal{T})\]
for some $n< \omega$ such that $n+1 \le k$ if $D^{\mathcal{T}} \cap b = \emptyset$. We set 
\[ \mathcal{Q}(b, \mathcal{T}) := \mathcal{J}^{\mathcal{M}_b}_{\gamma} \]
if there is such a $\gamma$, and let $\mathcal{Q}(b, \mathcal{T})$ be undefined else.
\end{Definition}
The next lemma is folklore, but we could not find a proof in the literature, this is why we include it here\footnote{We thank F. Schlutzenberg for several discussions on the topic which were very valuable.}.
\begin{Lemma}
Let $\mathcal{M}$ and $\mathcal{N}$ be $\omega$-sound premice which both project to $\omega$. Assume that $\mathcal{M}$ is an initial segment of $M_1$ and $\mathcal{N}$ is $\Pi^1_2$-iterable, and let $\Sigma$ denote the winning strategy for player II in $\mathcal{G}_{\omega} (\mathcal{M}, \omega_1+1)$. Then we can successfully compare $\mathcal{M}$ and $\mathcal{N}$ and consequentially $\mathcal{M} \triangleleft \mathcal{N}$ or $\mathcal{N} \trianglelefteq \mathcal{M}$.
\end{Lemma}
\begin{proof}
We compare $\mathcal{M}$ and $\mathcal{N}$ via iterating away the least disagreement, producing iteration trees $\mathcal{T}$ on the $\mathcal{M}$-side and $\mathcal{U}$ on the $\mathcal{N}$-side. We shall show that at every limit stage $\beta\le \omega_1$, there is a cofinal and wellfounded branch for each side and then use the usual finestructural argument to finish the proof. At limit stages, if $\mathcal{T}$ is the tree we have created so far on the $\mathcal{M}$-side, we pick the branch $b$ and $\mathcal{M}^{\mathcal{T}}_b$ according to the winning strategy $\Sigma$. 
We shall describe the strategy $\Gamma$ for player II on $\mathcal{N}$ which ensures a successful comparison of $\mathcal{M}$ and $\mathcal{N}$. For that it is sufficient to determine which branch $c$ player II should pick at limit length trees $\mathcal{U}$ such that the resulting model $\mathcal{M}^{\mathcal{U}}_c$ is wellfounded.

Assume that we arrive at a countable limit stage $\gamma=lh(\mathcal{T})=lh(\mathcal{U})$, where $\mathcal{T}$ and $\mathcal{U}$ denote the trees we have created so far and assume first that $\delta(\mathcal{T})=\delta(\mathcal{U})$, where $\delta(\mathcal{T})=sup \{lh(E^{\mathcal{T}}_{\alpha}) \, : \, \alpha < \gamma\}$.
As $\mathcal{M} \triangleleft M_1$, and the latter is 1-small, we know that $\mathcal{M}^{\mathcal{T}}_b$ is 1-small as well. Consequentially, there is a $\mathcal{Q}$-structure $\mathcal{Q}(b,\mathcal{T}) \trianglelefteq \mathcal{M}^{\mathcal{T}}_b$, which has to be of the form $L_{\alpha}(\mathcal{M}(\mathcal{T}))$, for some $\alpha < \omega_1$, as by 1-smallness you can not have another extender on the sequence until you destroy the Woodin cardinal. Let $\alpha$ be the ordinal height of that $\mathcal{Q}$-structure. Note that there is failure of $\delta(\mathcal{T})$ being Woodin which is definable over $L_{\alpha}(\mathcal{M}(\mathcal{T}))$ by definition. 

We shall now define which branch $c$ player II shall pick in $\mathcal{U}$.
As $\mathcal{N}$ is $\Pi^1_2$-iterable, we know that there is a branch $c$ through $\mathcal{U}$, such that $\alpha$ is in the wellfounded part of $\mathcal{M}^{\mathcal{U}}_c$ (or in a linear iteration of one $\mathcal{M}^{\mathcal{U}}_c$'s extenders, but we suppress this case as it will eventually result in exactly the same argument using the elementary embedding of the linear, iterated ultrapower).
As $\mathcal{T}$, $\mathcal{U}$ are the result of iterating away the least disagreement and $\delta(\mathcal{T})=\delta(\mathcal{U})$, we obtain that $\mathcal{M}^{\mathcal{T}}_b | \delta(\mathcal{T})= \mathcal{M}^{\mathcal{U}}_c | \delta(\mathcal{U})= \mathcal{M(\mathcal{T}})=\mathcal{M}(\mathcal{U})$, and the $\mathcal{Q}(b,\mathcal{T})$-structure $L_{\alpha}(\mathcal{M}({\mathcal{T}}))$ must also be contained in the wellfounded part of $\mathcal{M}^{\mathcal{U}}_c$. Hence there is a witness to $\delta(\mathcal{T})$ not being Woodin which is definable over $L_{\alpha}(\mathcal{M}({\mathcal{T}}))$, so the $\mathcal{Q}(c, \mathcal{U})$-structure must equal $  L_{\alpha}(\mathcal{M}({\mathcal{T}}))$ as long as $\mathcal{M}^{\mathcal{U}}_c \cap Ord> \alpha$ (if $\mathcal{M}^{\mathcal{U}}_c \cap Ord= \alpha$, we are done as $c$ will be wellfounded). This argument is uniform for the branch $c$ we pick so that $L_{\alpha}(\mathcal{M}({\mathcal{T}}))$ is also the $\mathcal{Q}(c',\mathcal{U})$-structure for any other sufficiently wellfounded branch $c'$ through $\mathcal{U}$.

We claim that $\mathcal{M}^{\mathcal{U}}_c$ is fully wellfounded. Indeed if we let $\alpha' > \alpha$ be countable, then there must be a branch $c'$ through $\mathcal{U}$ such that $c'$ is $\alpha'$-good. But now $c'=c$, as otherwise the two distinct branches $c$ and $c'$ through $\mathcal{U}$ 
will witness that $\delta(\mathcal{T})=\delta(\mathcal{U})$ is $A$-strong in $\mathcal{M}^{\mathcal{U}}_{c'}$ up to $\delta(\mathcal{T})$ for every $A \subset \delta(\mathcal{T})$, $A \in wfp(\mathcal{M}^{\mathcal{U}}_c) \cap wfp(\mathcal{M}^{\mathcal{U}}_{c'})$. But this contradicts the fact that there is a witness definable over $L_{\alpha}(\mathcal{M}_b^{\mathcal{T}})$ to the fact that $\delta(\mathcal{T})$ is not Woodin. This ends the discussion of the case where $\delta(\mathcal{T})=\delta(\mathcal{U})$ being countable.

We shall now consider the second case, where one side stops to use extenders at some point in the comparison process.
If it is the $\mathcal{N}$-side, which stops playing extenders at some point then there is nothing to show. If it is the $\mathcal{M}$-side who stops playing extenders, then there is a last model $\mathcal{M}_{\beta}$, while on the $\mathcal{N}$-side, $\mathcal{N}(\mathcal{U})$ exists and is contained in $\mathcal{M}_{\beta}$. But then there is a model of the form $L_{\alpha}(\mathcal{M}_{\beta}| \delta(\mathcal{U}))$ which codes a failure of $\delta(\mathcal{U})$ being Woodin, and the model can be used on the $\mathcal{N}$-side just as above to find the unique, wellfounded branch through $\mathcal{U}$.

What remains to show is that player II can find a wellfounded branch through $\mathcal{U}$
after $\omega_1$ many stages.
Assume that this is not the case, then after $\omega_1$-many rounds of the comparison game, there is a putative, $\omega$-maximal iteration tree $\mathcal{U}$ on $\mathcal{N}$ with no wellfounded branch.
We force with the L\'evy collapse $Col(\omega, \omega_1)$, and let $g$ be an $(M_1,Col(\omega, \omega_1))$-generic filter. As $\mathcal{N} \in M_1$ is $\Pi^1_2$-iterable, which is a $\Pi^1_2$-notion, we know that $\mathcal{N}$ is $\Pi^1_2$-iterable in $M_1[g]$ as well. Thus, in $M_1[g]$, there is a cofinal wellfounded branch for $\mathcal{U}$ and it has to be the unique one, by the argument above. Consequentially the branch $b$ is ordinal definable in $M_1[g]$ using $\mathcal{N}$, $\mathcal{M}$, $\mathcal{U}$, $ \mathcal{T}$ as parameters, and by the homogeneity of the L\'evy collapse, $b \in M_1$ as well, which is what we wanted to show.

To finish the proof, we have to show that $\mathcal{N} \trianglelefteq \mathcal{M}$ or vice versa. This is just a repetition of the finestructural fact that sufficiently iterable, $\omega$-sound and $\omega$-projecting premice line up (see \cite{Steel3}, Corollary 3.12).

As both premice must be countable, the above argument, together with the proof of the comparison lemma (see \cite{Steel3}, Theorem 3.11)  shows that we can successfully compare them in a $< \omega_1$-long process, while ensuring that the trees $\mathcal{T}$ on $\mathcal{M}$ and $\mathcal{U}$ on $\mathcal{N}$ have last models $\mathcal{M}_{\alpha}$ and $\mathcal{N}_{\beta}$ such that one is an initial segment of the other and such that the branch leading up to the shorter one does not drop in model or degree. Let us assume without loss of generality, that $\mathcal{M}_{\alpha} \trianglelefteq \mathcal{N}_{\beta}$ and $[0,\alpha]$ does not drop in model or degree. Then, as $\mathcal{M}$ projects to $\omega$, there are no extenders over $\mathcal{M}$ with critical point $<\rho_{\omega} (\mathcal{M})$, hence if $\alpha$ would be greater than 0, there would be a drop in model or degree in $[0, \alpha]$ which is a contradiction. So $\alpha=0$, and we turn to $\beta$. If $\beta \ne 0$, then, as $\rho_{\omega}(\mathcal{N})= \omega$, $\mathcal{N}_{\eta}$ is not $\omega$-sound, and consequentially $\mathcal{N}_{\beta}$ is not $\omega$-sound, hence $\mathcal{M}_{\alpha}$ is a proper initial segment of $\mathcal{N}_{\beta}$ and is countable there, as $\rho_{\omega}(\mathcal{M})=\omega$ and by soundness. But then the iteration from $\mathcal{N}$ leading up to $\mathcal{N}_{\beta}$ will not move $\mathcal{M}$, and $\mathcal{M}$ is an initial segment of $\mathcal{N}$, which is what we wanted.
 
\end{proof}

It is relatively straightforward to check that the set of reals which code $\Pi^1_2$-iterable, countable premice is itself a $\Pi^1_2$-definable set in the codes (see \cite{Steel2}, Lemma 1.7). Modulo the last lemma, this implies that there is a nice definition of a cofinal set of countable initial segments of $M_1$  in $\omega_1$-preserving
forcing extensions $M_1[G]$ of $M_1$, (in fact this definiton holds in all outer models of $M_1$ with the same $\omega_1$):

\begin{Lemma}
Let $M_1[G]$ be an $\omega_1$-preserving forcing extension of $M_1$. Then in $M_1[G]$ there is
$\Pi^{1}_2$-definable set $\mathcal{I}$ of premice  which are of the form $\mathcal{J}^{M_1}_{\eta}$ for 
some $\eta< \omega_1$. $\mathcal{I}$ is defined as
$$\mathcal{I}:= \{ \mouseM \text{ ctbl premouse} \, : \,\mouseM \text{ is } \Pi^{1}_2\text{-iterable}, \, 
\omega\text{-sound} \text{ and projects to } \omega \},$$
and the set
$$\{ \eta < \omega_1 \, : \, \exists \mouseN \in \mathcal{I} (\mouseN = \mathcal{J}^{M_1}_{\eta})\}$$ is cofinal in $\omega_1$.
\end{Lemma}
In particular $M_1| \omega_1$ is $\Sigma_1(\omega_1)$-definable in $\omega_1$-preserving generic extensions of $M_1$, as $x \in M_1 | \omega_1$ if and only if there is a transitive $U \models \ZFP$, $\omega_1 \subset U$, $\aleph_1^U=\aleph_1$ such that $U \models \exists \mathcal{M} \in \mathcal{I} \land x \in \mathcal{M}$, which suffices using Shoenfield absoluteness. A similar argument also shows that $\{ M_1 | \omega_1\}$ is $\Sigma_1(\omega_1)$ definable, as we can successfully compute it in transitive $\omega_1$-containing models, via the following $\Sigma_1(\omega_1)$-formula:
 \begin{align*}
(\ast) \quad X=M_1|\omega_1 \Leftrightarrow  \exists U (&U \text{ is a transitive model of } \ZFP \land \omega_1 \subset U
\land \\& U \models \forall \alpha < \omega_1 \exists r \in \mathcal{I} (\alpha \in (r \cap Ord)) \land \\& \, \,\qquad X \text{ is transitive and } X \cap Ord=\omega_1 \land 
\\& \qquad \qquad \forall x \in \mathcal{I} (x \subset X) \land \forall y \in X \exists x \in \mathcal{I} (y \in x))
\end{align*} 
Indeed, if the left hand side of $(\ast)$ is true, then any transitive $U$ which contains $M_1 | \omega_1$ as an element and which models $\ZFP$ will witness the truth of the right hand side, which is an immediate consequence of Shoenfield absoluteness.

If the right hand side is true, then, using the fact that $\Sigma^1_3$-statements are upwards absolute between $U$ and the real world, $U$ will contain an $\omega_1$-height, transitive structure $X$ which contains all countable initial segments of $M_1$, and such that every $y \in X$ is included in some element of $M_1| \omega_1$, in other words $X$ must equal $M_1 | \omega_1$.

We shall now apply the just obtained definability results to show that, once we repeat the coding procedure described in the earlier sections of the paper, with $M_1$ as the ground model, we can just use $\omega_1$ as parameter. For that it will be sufficient to show, that there are $\Sigma_1(\omega_1)$-definitions of an $\omega$-sequence of independent Suslin trees and a ladder system.

The first thing to note is that $M_1 | \omega_1$ can define a $\diamondsuit$-sequence in the same way as $L_{\omega_1}$ can. Indeed, as $M_1$ has a $\Delta_3^1$-definable wellorder of the reals whose definition relativizes to $M_1 | \omega_1$ we can repeat Jensen's original proof in $M_1$ to construct a candidate for the $\diamondsuit$-sequence, via picking at every limit stage $\alpha< \omega_1$ the $<_{M_1}$-least pair $(a_{\alpha}, c_{\alpha}) \in P(\alpha) \times P(\alpha)$ which witnesses that the sequence we have created so far is not a $\diamondsuit$-sequence. The proof that this defines already a witness for $\diamondsuit$ is finished as usual with a condensation argument. Hence we shall show that if $\mathcal{J}^{M_1}_{\beta}$ is least such that $(a_{\alpha} \, : \, \alpha< \omega_1) $ and $(A,C) \in \mathcal{J}^{M_1}_{\beta}$, where $(A,C)$ is the $<_{M_1}$-least witness for $(a_{\alpha})_{\alpha < \omega_1}$ not being a $\diamondsuit$-sequence, then there is an countable $N \prec \mathcal{J}^{M_1}_{\beta}$ such that the transitive collapse $\bar{N}$ is an initial segment of $M_1$. 

To see that in fact every such $N$ collapses to an initial of $M_1$, recall
the condensation result as in \cite{Steel3}, Theorem 5.1, which we can state in our situation as follows:
\begin{thm}
Let $\mathcal{M}$ be an initial segment of $M_1$. Suppose that $\pi: \bar{N} \rightarrow \mathcal{M}$ is the inverse of the transitive collapse and $crit(\pi)=\rho^{\bar{N}}_{\omega}$, then either
\begin{enumerate}
\item $\bar{N}$ is a proper initial segment of $\mathcal{M}$, or
\item there is an extender $E$ on the $\mathcal{M}$-sequence such that
$lh(E)=\rho^{\bar{N}}_{\omega}$, and $\bar{N}$ is a proper initial segment of $Ult_0(\mathcal{M},E)$.
\end{enumerate}
\end{thm}
We shall argue, that in our situation, the second case is ruled out, hence every $N \prec \mathcal{J}^{M_1}_{\beta}$ collapses to an initial segment of $M_1$. Indeed, due to the $\omega$-soundness of $\mathcal{J}^{M_1}_{\beta}$, every $N \prec \mathcal{J}^{M_1}_{\beta}$ will satisfy that\[ \rho_{\omega}^N= \rho_{\omega}^{\mathcal{J}^{M_1}_{\beta}}=\omega_1^{\mathcal{J}^{M_1}_{\beta}},\] hence $crit(\pi)= \omega_1^{\bar{N}}= \rho^{\bar{N}}_{\omega}$ by elementarity of $\pi$.

But $\bar{N} | \omega_1^{\bar{N}}= N | \omega_1^{\bar{N}}$, and as 
$\bar{N} | \omega_1^{\bar{N}}$ thinks that $\omega$ is its largest cardinal, 
$N | \omega_1^{\bar{N}}$ must believe this as well. But then there can not be an extender on the $N$-sequence which is indexed at $\omega_1^{\bar{N}}$, as otherwise $N | \omega_1^{\bar{N}}$ would think that $\omega_1^{\bar{N}}$ is inaccessible, which is a contradiction.
Hence, the condition $lh(E)=\rho_{\omega}^{\bar{N}}$ is impossible and all that is left is case 1, so $\bar{N}$ is an initial segment of $M_1$.

This shows that Jensen's construction of a $\diamondsuit$-sequence succeeds when applied to $M_1$. It is straightforward to verify that the recursive construction can be carried out in $M_1 | \omega_1$ by absoluteness. Consequentially the $\diamondsuit$-sequence is a $\Sigma_1$-definable class over $M_1 | \omega_1$.

We can use the $\diamondsuit$-sequence to construct an independent $\omega$-sequence of Suslin trees
due to a result of Jensen.

\begin{Definition}
 Let $T$ be a tree and $a \in T $ be a node, then $T_a$ denotes the tree $\{ x\in T \, : \, x>_Ta \}$.
 A Suslin tree $T$ is called full if for any level $\alpha$ and any finite sequence of nodes
 $a_0,...,a_n$ on the $\alpha$-th level of $T$, the tree $T_{a_0} \times T_{a_1}\times \cdot \cdot \cdot \times T_{a_n}$
 is a Suslin tree again.
\end{Definition}
A proof of the next result can be found in \cite{Handbook of topology}, Theorem 6.6.
\begin{thm}
 $\diamondsuit$ implies the existence of a full Suslin tree. Consequentially if $\diamondsuit$ holds
 then there is an $\omega$-length sequence of Suslin trees $\vec{T} =\{ T_0, T_1,... \}$ such that 
 any finite product of members of $\vec{T}$ is a Suslin tree again.
\end{thm}

The above proof, which is a refinement of Jensen's original construction of a Suslin tree in that one recursively picks (using the definable wellorder) at limit stages branches through $T$ which are generic for finite products over the least countable initial segment of $M_1$ which is able to see the construction up to that point, in fact relativizes down to $M_1 | \omega_1$, as the $\Delta^1_3$-definition of the wellorder of the $M_1$-reals can be applied inside $M_1 | \omega_1$, and the computation will always be correct. Hence, over $M_1 | \omega_1$ one can always define a full Suslin tree just as in $M_1$ and hence an $\omega$-sequence of independent Suslin trees. 
If we recall the proof of the main theorem, we see that the parameter $\vec{T}^0$ can be replaced by $M_1 | \omega_1$ as the latter can compute such an independent $\omega$-sequence of Suslin trees.

The second parameter in the statement of the theorem, namely the ladder system $\vec{C}$ can be replaced by $M_1 | \omega_1$ as well, for $M_1 | \omega_1$ can compute a canonical ladder system with the help of the $M_1$-wellorder of the reals. Thus, if we start with $M_1$ as the ground model, we end up with a universe where $\NS$ is saturated and $\Delta_1(M_1 | \omega_1)$-definable, and one can use $(\ast)$ to replace
the parameter $M_1 | \omega_1$ with just $\omega_1$.
We finally obtained a proof of the main theorem of this paper.

\begin{thm}
Suppose that the canonical inner model with one Woodin cardinal, $M_1$ exists. Then there is a model in which $\NS$ is saturated and $\Delta_1(\omega_1)$-definable.
\end{thm}
We end with a couple of open questions.

\begin{Question}

Is it consistent that $\NS$ is saturated, boldface $\Delta_1$-definable over $H(\omega_2)$ and $\MA$ does hold? How about other forcing axioms as $\MRP$?

\end{Question}

\begin{Question}
Is it consistent that $\NS$ is saturated, boldface $\Delta_1$-definable over $H(\omega_2)$ and a projective wellorder of the reals exists?
\end{Question}

\section*{Acknowledgement}

We thank the editor for his patience. Most of all we thank the amazing, anonymous referee who put a huge amount of effort into helping us with  turning an impressionistic sketch into a rigorous proof, pointing out numerous errors and inaccuracies.

\end{document}